\documentclass[12pt]{amsart}
\usepackage[psamsfonts]{amssymb}
\usepackage{amscd}
\def\R{\mathbb R}
\def\Z{\mathbb Z}

\def\N{\mathbb N}
\newtheorem{thm}{Theorem}[section]
\newtheorem{prop}[thm]{Proposition}
\newtheorem{cor}[thm]{Corollary}
\newtheorem{lemma}[thm]{Lemma}
\newtheorem*{sta}{A}
\newtheorem*{stb}{B}
\newtheorem*{stc}{C}
\newtheorem*{std}{D}
\newtheorem*{ste}{E}
\theoremstyle{remark}
\newtheorem*{remark}{Remark}
\numberwithin{equation}{section}
\begin{document}
%
%%%%%%%
\title[The structure of the conjugate locus]{The structure of the conjugate locus 
of a general point on ellipsoids and certain Liouville manifolds}
\author{Jin-ichi Itoh}
\address{Department of Mathematics, Faculty of Education,
Kumamoto University, Kumamoto 860-8555, Japan.}
\email{j-itoh@gpo.kumamoto-u.ac.jp}
\thanks{The first author was supported in part by Grant-in-Aid for Scientific 
Research (C)-23540098, (C)-26400072. }
\author{Kazuyoshi Kiyohara}
\address{Graduate School of Natural Science and Technology, Okayama University, Okayama 700-8530, Japan.}
\email{kiyohara@math.okayama-u.ac.jp}
\thanks{The second author was supported in part by Grant-in-Aid for Scientific 
Research (C)-23540089, (C)-26400071.}
\begin{abstract}It is well known since Jacobi that the geodesic flow of the ellipsoid is ``completely integrable'',
which means that the geodesic orbits are described in a certain explicit way. However, it does not directly indicate that any global behavior of the geodesics becomes easy to see. In fact, 
it happened quite recently that a proof for
the statement ``The conjugate locus of a general point in two-dimensional ellipsoid has just four cusps'' 
in Jacobi's Vorlesungen \"uber dynamik appeared in the literature. 

In this paper, we consider Liouville manifolds, a certain class of Riemannian manifolds which contains ellipsoids. We solve the geodesic equations; investigate the behavior of the Jacobi fields, especially the
positions of the zeros; and clarify the structure of the conjugate locus of a general point. In particular,
we show that the singularities arising in the conjugate loci are only cuspidal edges and $D_4^+$ Lagrangian
singularities, which would be the higher dimensional counterpart of Jacobi's statement.
\end{abstract}
\maketitle
\section{Introduction}
Let $M$ be a Riemannian manifold ($\dim M=n$) and let $\gamma(t)$ 
be a geodesic with $\gamma(0)=p\in M$. Then
$\gamma(t_1)$ $(t_1>0)$ is called the
{\it first conjugate point} of $p$ along $\gamma(t)$ if
$t=t_1$ is the largest value such that $\gamma|_{[0,t]}$ is the shortest one among the curves
which join $\gamma(0)$ and $\gamma(t)$
and which are ``infinitesimally close to $\gamma|_{[0,t]}$''.  
More precisely and more generally, the point $\gamma(T)$ $(T>0)$ is called 
a {\it conjugate point} of $p$
along the geodesic $\gamma(t)$ if there is a non-zero Jacobi field $Y(t)$ 
along $\gamma(t)$ such that $Y(0)=0$, $Y(T)=0$.
Conjugate points of $p$ along $\gamma(t)$ are
discrete in $t$; $\gamma(t_1)$, $\gamma(t_2)$, \dots $(0<t_1\le t_2\le\dots)$, called the first conjugate point,
the second conjugate point, etc..  The multiplicity is less 
than or equal to $n-1$. 

The $i$-th {\it conjugate locus} of $p\in M$ is the 
set of all $i$-th conjugate point of $p$ along the
geodesics emanating from $p$.
The term ``conjugate locus'' is usually used with the
meaning of the first conjugate locus.
For the generality of conjugate points and conjugate loci, one can 
refer to \cite{Kl},\cite{Sa2}. The simplest example of the conjugate locus is that of the sphere
of constant curvature $S^n$. In this case the first conjugate point of each $x\in S^n$ along
any geodesic is the 
antipodal point $-x$ and its multiplicity is $n-1$; thus the $i$-th conjugate locus ($1\le i\le n-1$) of $x$ is equal to $\{-x\}$, whereas the $j$-th conjugate locus
($n\le j\le 2n-2$) is equal to $\{x\}$.

To understand the global behavior of the geodesics, it is crucial to know the structure
of conjugate loci and cut loci of points. In general, however, it is quite difficult to
determine conjugate loci and cut loci explicitly, except for symmetric spaces and 
other few examples (see \cite[Introduction]{IK2}). In the previous papers \cite{IK1} and \cite{IK3} we determined
the structure of conjugate loci and cut loci of points for the tri-axial ellipsoid and certain
Liouville surfaces. Also we determined in \cite{IK2} the structure of cut loci of points for
the ellipsoid and certain Liouville manifolds of dimension greater than two. In this paper 
we clarify the structure of the conjugate locus of a general point on the ellipsoid and certain Liouville manifolds of dimension greater than two. In particular, we give a detailed description
for the singular points on the conjugate locus. This would be a higher-dimensional counterpart 
of ``the last geometric statement of Jacobi'', which says that the conjugate locus of a
non-umbilic point of the two-dimensional ellipsoid has exactly four cusps (\cite{J}, \cite{JW}; see also \cite{IK1}, \cite{IK3}, \cite{Si}, \cite{ST1}).

Now, let us illustrate our results in detail by taking the ellipsoid 
$M: \sum_{i=0}^n  u_i^2/a_i=1$ 
($0<a_n<\dots<a_0$) as an example. {\it The elliptic coordinate system} $(\lambda_1,\dots,\lambda_n)$ 
on $M$ $(\lambda_n\le \dots\le\lambda_1)$ is defined by the following 
identity in $\lambda$:
\begin{equation*}
\sum_{i=0}^n\frac{u_i^2}{a_i-\lambda}-1=\frac{\lambda\prod_{k=1}^n(\lambda_k-\lambda)}{\prod_i (a_i-\lambda)}\ .
\end{equation*}
For a fixed $u\in M$, $\lambda_k$ are determined by $n$ 
``confocal quadrics'' passing through $u$. 
From $\lambda_k$'s, $u_i$ are explicitly described as
\begin{equation*}
u_i^2=\frac{a_i\prod_{k=1}^n(\lambda_k-a_i)}{\prod_{j\ne i} (a_j-a_i)}\ ,
\end{equation*}
and the range of $\lambda_k$ is $a_k\le \lambda_k\le  a_{k-1}$. Also the metric $g$ is described
as
\begin{equation}
g=\sum_{i=1}^n\frac{(-1)^n\lambda_i\,\prod_{l\ne i}(\lambda_l-\lambda_i)}
{4\prod_{j=0}^n(\lambda_i-a_j)}\,d\lambda_i^2
\end{equation}

Let $N_k$ be the ellipsoid of codimension one in $M$ defined by
\begin{gather*}
N_k=\{u=(u_0,\dots,u_n)\in M\ |\ u_k=0\ \}\qquad(0\le k\le n)\,,
\end{gather*}
which is a totally geodesic submanifold of $M$. By the elliptic coordinates the 
submanifold $N_k$
is expressed as 
\begin{equation*}
N_k=\{\lambda_k=a_k\quad\text{or}\quad\lambda_{k+1}=a_k\ \} .
\end{equation*}
We shall say that $p\in M$ is a {\it general point} if $p\not\in N_k$ for any $k$.
%($1\le k\le n-1$). Note that we do not require $p\not\in N_0$ nor $p\not\in N_n$.

Now, let $p\in M$ be a general point.
Each element $v$ of the tangent space $T_pM$ at $p$
%If the point $p$ moreover satisfies $p\not\in N_0\cup N_n$, then 
is expressed as:
\begin{equation*}
v=\sum_{i=1}^n \,v_i\,\frac{\partial}{\partial \lambda_i}\,.
\end{equation*}
Then, putting 
\begin{equation*}
\tilde v_i=\sqrt{\frac{(-1)^{n-i}\,\lambda_i\prod_{l\ne i}(\lambda_l-\lambda_i)}
{(-1)^i\,4\prod_{j=0}^n(\lambda_i-a_j)}}\,v_i\,,
\end{equation*}
we have an Euclidean coordinate system $(\tilde v_1,\dots,\tilde v_n)$ on $T_pM$, i.e.,
\begin{equation*}
g(v,v)=1\quad\text{if and only if}\quad \sum_{i=1}^n \tilde v_i^2=1\,.
\end{equation*}
%If $p\in N_0$ or $p\in N_n$, then, by taking an appropriate sign of the above square root,
%we also obtain the Euclidean coordinates $(\tilde v_1,\dots,\tilde v_n)$ on $T_pM$. 
We now define an elliptic coordinate system $(\mu_1,\dots,\mu_{n-1})$ on the unit tangent space 
$U_pM\subset T_pM$ by the following identity in $\mu$:
\begin{equation*}
\sum_{i=1}^n\frac{\tilde v_i^2}{\mu-\lambda_i(p)}=
\frac{\prod_{k=1}^{n-1}(\mu-\mu_k)}{\prod_{j=1}^n(\mu-\lambda_j(p))},\quad 
\lambda_{i+1}(p)\le \mu_i\le \lambda_i(p)\,.
\end{equation*}
Then
\begin{equation*}
\tilde v_i^2=\frac{\prod_{k=1}^{n-1}(\lambda_i(p)-\mu_k)}{\prod_{j\ne i}(\lambda_i(p)-\lambda_j(p))},\qquad \sum_{i=1}^{n}\tilde v_i^2=1\,.
\end{equation*}
Define the submanifolds (with boundary) $C_i^\pm$ $(1\le i\le n-1)$ of $U_pM$ by
\begin{equation*}
 C_i^-=\{v\in U_pM\,|\, \mu_i(v)=\lambda_{i+1}(p)\},\quad 
 C_i^+=\{v\in U_pM\,|\, \mu_i(v)=\lambda_{i}(p)\}\,.
\end{equation*}
It is seen that $C_{i-1}^-\cup C_i^+$ is equal to the great sphere $\tilde v_i=0$ and
they are diffeomorphic to
\begin{gather*}
C_i^-\simeq  S^{i-1}\times \bar D^{n-1-i},\quad C_i^+\simeq \bar D^{i-1}\times S^{n-1-i}\,,
\end{gather*}
where $S^k$ and $\bar D^k$ stand for the $k$-sphere and the closed $k$-disk ($S^0$
and $\bar D^0$ stand for the set of two points and that of one point) respectively.
Also for the boundary $\partial C_i^\pm$ of $C_i^\pm$, 
\begin{gather*}
\partial C_i^+=\partial C_{i-1}^-= C_i^+\cap C_{i-1}^-\simeq S^{i-2}\times S^{n-1-i}\quad (2\le i\le n-1)\,,\\
\partial C_{n-1}^-=\emptyset=\partial C_1^+\,.
\end{gather*}

Put $V_i=\pm (\partial/\partial \mu_i)/\|\partial/\partial \mu_i\|$ $(1\le i\le n-1)$.
One can see that at each point $v\in U_pM-\partial C_i^\pm$, the vector field $V_i$ is smoothly
defined on a neighborhood of the point by taking the appropriate sign. 
Let $\gamma_v(t)$ be the geodesic on $M$ with the initial vector $\dot\gamma_v(0)=v\in U_pM$ and
let $Y_i(t,v)$ $(1\le i\le n-1)$ be the Jacobi field along the geodesic $\gamma_v(t)$ defined by 
the initial data $Y_i(0,v)=0$, $Y'_i(0,v)=V_i(v)$
(``prime'' represents the covariant derivative in $t$). Assume first that $v\not\in \partial C_j^\pm$ for any $j$.
Then, as was already shown in \cite[Proposition 5.1]{IK2}, the Jacobi field $Y_i(t,v)$ is of the form
\begin{equation*}
Y_i(t,v)=y_i(t,v)\tilde V_i(t, v),
\end{equation*}
where $y_i(t,v)$ is a function and $\tilde V_i(t,v)$ is the parallel vector field along the 
geodesic $\gamma_v(t)$ such that $\tilde V_i(0,v)=V_i(v)$. 
(Actually, we may say $\tilde V_i(t,v)=V_i(\dot\gamma_v(t))$.) Let $t=r_i(v)$ be the first zero of the 
function $t\mapsto y_i(t,v)$ for $t>0$. It turns out that the function $r_i(v)$ can be
continuously extended to all over $U_pM$ and is of $C^\infty$ outside $\partial C_i^\pm$. Then our first result is the following
\begin{sta}[Proposition~\ref{prop:conj}]
\begin{enumerate}
\item $r_{n-1}(v)\le r_{n-2}(v)\le \dots\le r_1(v)$ for any $v\in U_pM$.
\item $r_{i-1}(v)=r_{i}(v)$ if and only if $v\in \partial C_{i-1}^-=\partial C_i^+$\quad$(2\le i\le n-1)$
\end{enumerate}
\end{sta}
Put
\begin{equation*}
\tilde K_i(p)=\{r_i(v)v\,|\,v\in U_pM\},\quad K_i(p)=\{\gamma_v(r_i(v))\,|\,v\in U_pM\}\,.
\end{equation*}
As a consequence of the above proposition, we have 
\begin{stb}[Theorem~\ref{thm:conj}]\label{thm2}
\begin{enumerate}
\item $K_{n-1}(p)$ is the (first)
conjugate locus of $p$.
\item If $M$ is close to the round sphere in an appropriate sense, then $K_{n-i}(p)$
is the $i$-th conjugate locus of $p$ for $2\le i\le n-1$.
\end{enumerate}
\end{stb}
In this case, $\tilde K_{n-1}(p)$ is called the {\it tangential conjugate locus} of $p$, and 
under the situation of (2) $\tilde K_{n-i}(p)$ is called the $i$-th tangential 
conjugate locus.
The assumption in (2) of the above theorem is actually given as follows: ``if the second zero,
say $r_{n-1}^2(v)$, of $y_{n-1}(t,v)$ is greater than $r_1(v)$ for any $v\in U_pM$''.

Next, let us explain our results on the ``singular points'' of the conjugate locus.
Define the map $\Phi: U_pM\to M$ by 
\begin{equation*}
\Phi(v)=\text{Exp}_p(r_{n-1}(v)v)=\gamma_v(r_{n-1}(v))\,,
\end{equation*}
whose image is the conjugate locus $K_{n-1}(p)$ of $p$. Then we have
\begin{stc}[Theorem~\ref{thm:cusp}]
\begin{enumerate}
\item $\Phi$ is an immersion outside $C_{n-1}^-\cup C_{n-1}^+$.
\item The germ of $\Phi$ is a cuspidal edge at each point of $C_{n-1}^-$ and each interior
point of $C_{n-1}^+$; the restriction of $\Phi$ to (the interior of) $C_{n-1}^\pm$ are 
immersions to the edges of the vertices.
\end{enumerate}
\end{stc}
We note that, as was shown in \cite[Theorem~7.1]{IK2}, the restriction of $\Phi$ to 
$C_{n-1}^-$ is actually an embedding and the image bounds the cut locus of $p$.

As for the singularities arising on the boundary $\partial C_{n-1}^+$, we need to treat them
as the singularities of the map Exp$_p: T_pM\to M$, since the map $\Phi$ is not
differentiable at points on $\partial C_{n-1}^+$ (and the tangential conjugate locus 
$\tilde K_{n-1}(p)$ is not smooth at $r_{n-1}(v)v$, $v\in\partial C_{n-1}^+$).
However, it should be noted that the function $r_{n-1}$ restricted to $\partial C_{n-1}^+$
is smooth. Thus $\tilde S=\{r_{n-1}(v)v\,|\,v\in \partial C_{n-1}^+\}$ is a submanifold of $T_pM$
diffeomorphic to $S^{n-3}\times S^0$.
\begin{std}[Corollary~\ref{cor1}]
The germ of the map {\rm Exp}$_p: T_pM\to M$ at each point $w\in \tilde S$ is 
a $D_4^+$ Lagrangian singularity. 
\end{std}
The notion of $D_4^+$ Lagrangian singularity first appeared in the work of Arnold \cite{A}, where he
classified the ``simple Lagrangian singularities'' (see also \cite{AGV}). We shall give its precise description in \S{7}. Here we only note one consequence of the above result: The singularity
at each point of $\tilde S\subset \tilde K_{n-1}(p)$ is
a cone-edge, i.e., there
is a coordinate system $(w_1,\dots,w_n)$ on $T_{p}M$ around the point (represented by the origin)
such that $\tilde K_{n-1}(p)$ and $\tilde S$ are described as
\begin{gather*}
\tilde K_{n-1}(p):\quad w_1^2+w_2^2=w_3^2,\quad w_3\le 0\,,\\
\tilde S:\quad w_1=w_2=w_3=0\,.
\end{gather*}
and this cone-edge is connected to $\tilde K_{n-2}(p)$ as
\begin{equation*}
\tilde K_{n-2}(p):\quad w_1^2+w_2^2=w_3^2,\quad w_3\ge 0\,.
\end{equation*}
This illustrates the zeros of the smooth function 
$\det\,d(\text{Exp}_p)$ near $r_{n-1}(v)v$, where $v\in \partial C_{n-1}^+$.

Under the situation of the statement {\bf B} (2), we have the similar result for $K_{n-i}(p)$.
\begin{ste}[Theorem~\ref{thm:cusp}, Corollary~\ref{cor2}]
Suppose $r_{n-1}^2(v)$ is greater than $r_1(v)$ for any $v\in U_pM$. Then, defining
the map $\Phi_{n-i}:U_pM\to M$ by
\begin{equation*}
\Phi_{n-i}(v)=\text{\rm Exp}_p(r_{n-i}(v)v)\,,
\end{equation*} 
we have:
\begin{enumerate}
\item $\Phi_{n-i}$ is an immersion outside $C_{n-i}^-\cup C_{n-i}^+$.
\item The germ of $\Phi_{n-i}$ is a cuspidal edge at each interior point of $C_{n-i}^-$ and 
$C_{n-i}^+$; the restriction of $\Phi_{n-i}$ to the interior of $C_{n-i}^\pm$ is 
immersions to the edges of the vertices.
\item The germ of the map {\rm Exp}$_p: T_pM\to M$ at each point $r_{n-i}(v)v$, $v\in \partial C_{n-i}^\pm$, 
is a $D_4^+$ Lagrangian singularity and the restriction {\rm Exp}$_p|_{\partial C_{n-i}^\pm}$
is an immersion to the edge of vertices.
\end{enumerate}
\end{ste}

The present paper is partly a continuation of our
previous paper \cite{IK2}, where we studied the cut loci of
points on certain Liouville manifolds diffeomorphic to the sphere.
Each Liouville manifold which is considered here is,  as in \cite{IK2}, defined with 
$n+1$ constants $a_0>\dots>a_n>0$ 
and a positive function $A(\lambda)$ on $[a_n, a_0]$ ($A(\lambda)=\sqrt{\lambda}$ in the case of the ellipsoid).
This function $A(\lambda)$ is assumed to satisfy a certain monotonicity condition, which
is a bit stronger than the condition (4.1) posed in \cite{IK2}. We shall 
explain this condition in \S4.  In \S2 and \S3 
we give a brief summary of Liouville manifolds and the behavior of geodesics on them. Since they are almost the same as those in \cite{IK2}, we omit the proofs there. 

Under the condition given in \S4, we shall 
investigate the positions of zeros of Jacobi fields in detail in \S5 and 
describe the structure of the conjugate locus of a 
general point in \S6. In particular, we shall show in this section that
the major part of the singularities of the conjugate locus of a general point are cuspidal edges. In \S7 we shall investigate the remaining singularities,
which appear as the end points of the cuspidal edges and which are also
points of double conjugacy. We shall show that those are $D_4^+$ Lagrangian singularities.
Also, as an application of the 
results obtained in \S5, we shall illustrate there 
an interesting asymptotic nature of the distribution 
of zeros of Jacobi fields.

The authors are grateful to Shuichi Izumiya and Kentaro Saji for their helpful comments
on the $D_4^+$ Lagrangian singularity.

\subsection*{Preliminary remarks and notations}
In this paper the geodesics will be described in the Hamiltonian formalism. Therefore the geodesic flow is described in the cotangent bundle.
Let $M$ be a Riemannian manifold and $g$ its Riemannian 
metric.
By $\flat:TM\to T^*M$ we denote the bundle isomorphism
determined by $g$ (Legendre transformation). We also use the
symbol $\sharp=\flat^{-1}$.
The canonical 1-form on the cotangent bundle $T^*M$ is denoted by $\alpha$.
For a canonical coordinate system $(x,\xi)$ on $T^*M$ ($
x$ being a coordinate system on $M$), $\alpha$ is expressed as
$\sum_i\xi_idx_i$. Then the 2-form $d\alpha$ 
represents the standard
symplectic structure on $T^*M$.

Let $E$ be the function on $T^*M$ defined by 
\begin{equation*}
E(\lambda)=\frac12g(\sharp(\lambda),\sharp(\lambda))=
\frac12\sum_{i,j}g^{ij}(x)\xi_i\xi_j,\quad \lambda=(x,\xi)\in T^*M\,.
\end{equation*}
We call it the (kinetic) energy function of $M$.
For a function $F, H$ on $T^*M$, the Hamiltonian vector field 
$X_F$ and the Poisson bracket $\{F,H\}$ are defined by
\begin{equation*}
X_F=\sum_i\left(\frac{\partial F}{\partial \xi_i}\frac{\partial}
{\partial x_i}-\frac{\partial F}{\partial x_i}\frac{\partial}
{\partial \xi_i}\right)\, ,\qquad
\{F,H\}=X_FH\,.
\end{equation*}
Then $X_E$ generates the geodesic flow $\{\zeta_t\}_{t\in\R}$, i.e., each curve $\gamma(t)=
\pi(\zeta_t\lambda)$ $(\lambda\in T^*M)$ is a geodesic of the Riemannian manifold
$M$, where $\pi:T^*M\to M$ is the bundle projection. 
In this case we have $\flat(\dot\gamma(t))=\zeta_t\lambda$.
\section{Liouville manifolds}
Liouville manifold is, roughly speaking, a class of Riemannian manifold whose geodesic equations are ``integrated
in the same way as those of ellipsoids''.
The precise definition is as follows.
Let $M$ be a Riemannian manifold of dimension $n$
and let $\mathcal F$ be an $n$-dimensional vector space
of functions on the cotangent bundle $T^*M$.
Then the pair
$(M,\mathcal F)$ is called a Liouville manifold
if: i) each $F\in \mathcal F$ is
fiberwise a homogeneous quadratic polynomial; ii) those quadratic forms are
simultaneously normalizable on each fiber; iii) $\mathcal F$ is commutative with respect to the Poisson bracket; iv)
$\mathcal F$ contains the energy function $E$ (the Hamiltonian of the geodesic flow); and
v) $\{F|_{T^*_pM}\,|\,F\in\mathcal F\}$ is $n$-dimensional at some point $p\in M$.
For the generality of Liouville manifolds, we refer to \cite{Ki2}.

As in \cite{IK2}, we treat in this paper a subclass of ``compact Liouville
manifolds of rank one and type (A) (cf. \cite{Ki2})''.
An explanation of this subclass and the
geodesic equations on it were already given in
\cite{IK2}. We shall briefly illustrate it
in this and the next sections (without proof) for the sake of convenience.

Each Liouville manifold treated here is constructed
from $n+1$ constants $a_0>\cdots > a_{n}>0$ and a positive $C^\infty$
function $A(\lambda)$ on the closed interval $a_n\le \lambda\le a_0$. 
Let $\alpha_1,\dots,\alpha_n$ be positive numbers defined by
\begin{equation*}
\alpha_i=2\int_{a_i}^{a_{i-1}}\frac{A(\lambda)\ d\lambda}
{\sqrt{(-1)^i\prod_{j=0}^n (\lambda-a_j)}}\qquad
(i=1,\dots,n)\,.
\end{equation*}
Define the $C^\infty$ function $f_i$ on the circle $\R/\alpha_i\Z=
\{x_i\}$ $(1\le i\le n)$
by the conditions:
\begin{gather}\label{eq:fdiff}
\left(\frac{df_i}{dx_i}\right)^2=\frac{(-1)^i4\prod_{j=0}
^n (f_i-a_j)}{A(f_i)^2}\\
f_i(0)=a_i,\ f_i(\frac{\alpha_i}4)=a_{i-1},\quad
f_i(-x_i)=f_i(x_i)=f_i(\frac{\alpha_i}2-x_i)\ .
\end{gather}
Then the range of $f_i$ is $[a_i,a_{i-1}]$,
and $f_i$ actually has the period $\alpha_i/2$.  

Put
\begin{equation*}
R=\prod_{i=1}^n(\R/\alpha_i\Z)\ .
\end{equation*}
Let $\tau_i$ $(1\le i\le n-1)$ be the involutions
on the torus $R$ defined by
\begin{equation*}
\tau_i(x_1,\dots,x_n)=(x_1,\dots,x_{i-1},-x_i,
\frac{\alpha_{i+1} }2-x_{i+1},x_{i+2},\dots,x_n) \ ,
\end{equation*}
and let $G$ $(\simeq (\Z/2\Z)^{n-1})$ be the group of transformations generated by $\tau_1$, $\dots$,
$\tau_{n-1}$. Then the quotient space
$M=R/G$ is homeomorphic to the $n$-sphere, and 
moreover, $M$ has a unique differentiable structure so that the quotient map $R\to M$ is
of $C^\infty$ and the symmetric 2-form $g$ given by
\begin{equation}\label{eq:metric}
g=\sum_i(-1)^{n-i}\left(\prod_{l\ne i}(f_l(x_l)-f_i(x_i))\right)\,dx_i^2
\end{equation}
represents a $C^\infty$ Riemannian
metric on $M$. We regard $M$ as a
Riemannian manifold with this metric $g$. 
As a result, $M$ is diffeomorphic
to the $n$-sphere $S^n$.

Now, put
\begin{equation*}
b_{ij}(x_i)=
\begin{cases}
(-1)^i\prod_{\substack{1\le k\le n-1\\ k\ne j}}(f_i(x_i)-a_k)\quad (1\le j\le n-1)\\
(-1)^{i+1}\prod_{k=1}^{n-1}(f_i(x_i)-a_k)\qquad (j=n)
\end{cases}\ ,
\end{equation*}
and define functions $F_1$, $\dots$, $F_{n-1}$, $F_n=2E$ on the cotangent
bundle by
\begin{equation}\label{integrals}
\sum_{j=1}^n b_{ij}(x_i)F_j(x,\xi)=\xi_i^2\quad (1\le i\le n) ,
\end{equation}
where $\xi_i$ are the fiber coordinates with respect to
the base coordinates $(x_1,\dots,x_n)$. Then $F_i$ represent well-defined $C^\infty$
functions on $T^*M$. 

Computing the inverse matrix of $(b_{ij})$
explicitly, we have
\begin{align*}
2E=&\sum_{i=1}^n\frac{(-1)^{n-i} \xi^2_i}{\prod_{\substack{1\le l\le n\\l\ne i}}
(f_l(x_l)-f_i(x_i))}\\
F_j=&\frac1{\prod_{\substack{1\le k\le n-1\\k\ne j}}(a_k-a_j)}
\sum_{i=1}^n\frac{(-1)^{n-i}\prod_{\substack{1\le l\le n\\l\ne i}}(f_l(x_l)-a_j)}
{\prod_{\substack{1\le l\le n\\l\ne i}}(f_l(x_l)-f_i(x_i))}\ \xi^2_i\\
& (1\le j\le n-1)\ .
\end{align*}
Therefore $E$
is the energy function, i.e., the Hamiltonian of the associated geodesic flow of $M$.
From the formula (\ref{integrals}) one can easily see 
that
\begin{equation*}
\{F_i,F_j\}=0\qquad (1\le i,j\le n)\ ,
\end{equation*}
where $\{,\}$ denotes the Poisson bracket (see 
\cite[Prop.\,1.1.3]{Ki2}). Thus, denoted by $\mathcal F$ the vector space spanned by $F_1,\dots,
F_n$, the pair $(M, \mathcal F)$ becomes a
Liouville manifold.

The following proposition is obvious.
\begin{prop}\label{prop:isometry}
For each $i$, the map 
\begin{equation*}
x\mapsto (\dots,x_{i-1},-x_i,x_{i+1},\dots)\quad\text{or}\quad x
\mapsto (\dots,x_{i},\frac{\alpha_{i+1}}2-x_{i+1},x_{i+2},\dots)
\end{equation*}
defines an isometry of $M$ which preserves $F_j$ for any $j$. This map is the symmetry with respect to
$N_i$.
\end{prop}

As examples, if $A(\lambda)$ is a constant function, then
$M$ is the sphere of constant curvature. This case is
explained in detail in \cite[pp.71--74]{Ki2}.  If $A(\lambda)
=\sqrt{\lambda}$, then $M$ is isometric to the ellipsoid
$\sum_{i=0}^n{u_i^2}/{a_i}=1$. In this case, the system 
of functions
$(f_1(x_1),\dots,f_n(x_n))$ 
is nothing but the elliptic coordinate system (see \cite[p.261]{IK2}).

The manifold $M$ has some special submanifolds:  Put
\begin{gather*}
N_k=\{x\in M\ |\ f_k(x_k)=a_k \quad\text{or}\quad
f_{k+1}(x_{k+1})=a_k\}\quad(0\le k\le n) ,\\
J_k=\{x\in M\ |\ f_k(x_k)=
f_{k+1}(x_{k+1})=a_k\}\quad (1\le k\le n-1) .
\end{gather*}
Then we have, putting $(F_k)_p=F_k\vert_{T_p^*M}$,
\begin{prop}\label{prop:subm}
\begin{itemize}
\item[(1)] $J_k=\{p\in M\ |\ (F_k)_p=0\}$.
\item[(2)] $N_k=\{p\in M\ |\ \text{\rm rank }(F_k)_p\le 1\}$
\quad $(1\le k\le n-1)$.
\item[(3)] $\bigcup_k J_k$ is identical with the branch
locus of the covering $R\to M=R/G$.
\item[(4)] $N_k$ is a totally geodesic submanifold of codimension one \ $(0\le k\le n)$.
\item[(5)] $J_k\subset N_k$, and $J_k$ is diffeomorphic to
$S^{k-1}\times S^{n-k-1}$.
\end{itemize}
\end{prop}
In the case of the ellipsoid, $N_k$ is equal to the intersection with the cartesian hyperplane $u_k=0$,
and the set $\bigcup_k J_k$ is identical with the locus where some principal curvature has
multiplicity two.
%%%
%%%
\section{Geodesic equations}
Suppose that $c=(c_1,\dots,c_{n-1},1)$ is a
regular value of the map
\begin{equation*}
\boldsymbol F=(F_1,\dots, F_{n-1},2E): T^*M\to \R^n\ ,
\end{equation*}
then its inverse image is a disjoint union of tori, and
the vector fields $X_{F_j}$, $X_E$ on it are
mutually commutative and linearly independent
everywhere. Here $X_f$ denotes the Hamiltonian vector field
determined by a function $f$; 
\begin{equation*}
X_f=\sum_i\left(\frac{\partial f}{\partial \xi_i}
\frac{\partial }{\partial x_i}-
\frac{\partial f}{\partial x_i}
\frac{\partial }{\partial \xi_i}\right)\ .
\end{equation*}

Let $\omega_j$ $(1\le j\le n)$ be
the dual $1$-forms of $\{\pi_*X_{F_j}\}$, where
$\pi:T^*M\to M$ is the bundle projection. Then, by 
(\ref{integrals}) we have
\begin{equation*}
\omega_l=\sum_i\frac{b_{il}}{2\xi_i}\ dx_i\qquad
(1\le l\le n).
\end{equation*}
They are closed $1$-forms, and the geodesic
orbits are determined by
\begin{equation}\label{eq:geod0}
\omega_l=0 \qquad (1\le l\le n-1),
\end{equation}
and the length parameter $t$ on an orbit
is given by
\begin{equation}\label{eq:length}
dt=2\omega_n.
\end{equation}
Thus the geodesics are described with the integration of closed 1-forms
which contains $c_i$'s as parameters.

To observe the behavior of geodesics it is more convenient to use the
constants $b_1,\dots,b_{n-1}$ defined below than using $c_i$'s as parameters:
Put
\begin{equation*}
\Theta(\lambda)=\sum_{j=1}^{n-1}\left(
\prod_{\substack{1\le k\le n-1\\ k\ne j}}
(\lambda-a_k)\right)\,c_j-\prod_{k=1}^{n-1}(\lambda-a_k)\ .
\end{equation*}
If a unit covector $(x,\xi)\in U^*M$ with every $\xi_i\ne 0$ lies on $\boldsymbol F^{-1}(c)$, then we have, by (\ref{integrals}),
\begin{equation*}
(-1)^i\Theta(f_i(x_i))=\xi_i^2> 0\ .
\end{equation*}
Therefore the algebraic equation $\Theta(\lambda)=0$ has $n-1$ real roots in this case.
It thus follows by continuity that for each $c\in \boldsymbol F(U^*M)$ there are
constants $b_1\ge \dots\ge b_{n-1}$ such that
\begin{gather*}
\Theta(\lambda)=-\prod_{i=1}^{n-1}(\lambda-b_i)\ ,\\
f_i(x_i)\ge b_i\ge f_{i+1}(x_{i+1})\qquad(1\le i\le n-1)\ .
\end{gather*}

Note that the range of $b_i$'s are given by
\begin{equation}\label{brange}
a_{i+1}\le b_i\le a_{i-1}\ ,\qquad b_{i+1}\le b_i\ .
\end{equation}
In fact, it can be verified that for each $b_i$'s satisfying \eqref{brange}
there is a unit covector $\mu\in U^*M$ such that
$F_i(\mu)=c_i$ ($1\le i\le n-1$), where
\begin{equation}\label{b-c}
c_i=\frac{-\prod_{l=1}^{n-1}(a_i-b_l)}
{\prod_{\substack{1\le k\le n-1\\ k\ne i}}
(a_i-a_k)}
\qquad (1\le i\le n-1)\ .
\end{equation}
Note also that if
$b_1$, $\dots$, $b_{n-1}$ satisfy 
\begin{equation}\label{b-cond}
a_{i+1}< b_i< a_{i-1}\ ,\quad b_i\ne a_i,\quad 
b_{i+1}< b_i\qquad \text{for any }i
\end{equation}
then the corresponding $c=(c_1,\dots,c_{n-1},1)$ is a regular
value of $\boldsymbol F$.

Now, put
\begin{align*}
&a_i^+=\max\{a_i,b_i\}\quad (1\le i\le n-1),
\quad a_n^+=a_n\\
&a_i^-=\min\{a_i,b_i\}\quad (1\le i\le n-1),
\quad a_0^-=a_0\ .
\end{align*}
If $b_1,\dots,b_{n-1}$ satisfy the condition (\ref{b-cond}), then
the $\pi$-image of a connected component of
$\boldsymbol F^{-1}(c)$ (a Lagrange torus) is of 
the form
\begin{equation*}
L_1\times\dots \times L_n\subset M\ ,
\end{equation*}
where each $L_i$ is a connected component of the inverse
image of $[a_i^+,a_{i-1}^-]$ by the map
\begin{equation*}
f_i:\R/\alpha_i\Z\to [a_i,a_{i-1}]\ .
\end{equation*}
(Precisely speaking, $L_1\times\dots\times 
L_n\subset R$; but it is injectively mapped into $M$
by the branched covering $R\to M$.)
Along a corresponding geodesic, the coordinate function $x_i(t)$ moves on $L_i$ and $f_i(x_i(t))\in [a_i^+,a_{i-1}^-]$.

After all, the equations of geodesic orbits
\begin{equation*}
\omega_l=0\qquad (1\le l\le n-1)
\end{equation*}
are described as
\begin{equation*}
\sum_{i=1}^n\frac{\epsilon_i(-1)^i\prod_{\substack{1\le k\le n-1\\ k\ne l}}(f_i(x_i)-a_k)\ dx_i}
{\sqrt{(-1)^{i-1}\prod_{k=1}^{n-1}(f_i(x_i)-b_k)}}
=0\qquad (1\le l\le n-1)\ ,
\end{equation*}
where $\epsilon_i=\text{sign}\,(\xi_i)=\text{sign}\,(dx_i/dt)$.
This system of equations is equivalent to
\begin{equation}\label{geodeq}
\sum_{i=1}^n\frac{\epsilon_i(-1)^iG(f_i)\ dx_i}
{\sqrt{(-1)^{i-1}\prod_{k=1}^{n-1}(f_i-b_k)}}
=0
\end{equation}
for any polynomial $G(\lambda)$ of degree $\le n-2$.
Thus by \eqref{geodeq} we have
\begin{equation}\label{geodeqint}
\sum_{i=1}^n\int_s^{t}\frac{(-1)^iG(f_i)}
{\sqrt{(-1)^{i-1}\prod_{k=1}^{n-1}(f_i-b_k)
}}\ 
\left|\frac{dx_i(t)}{dt}\right|
\ dt=0
\end{equation}
for any period $[s, t]$, where $f_i=f_i(x_i(t))$.
By (\ref{eq:fdiff}) those equations are also described as
\begin{equation}\label{geodeqintf}
\sum_{i=1}^n\int_s^{t}\frac{(-1)^iG(f_i)A(f_i)}
{\sqrt{-\prod_{k=1}^{n-1}(f_i-b_k)
\cdot\prod_{k=0}^n(f_i-a_k)}}\ 
\left|\frac{df_i(x_i(t))}{dt}\right|
\ dt=0\ .
\end{equation}
%By using the variables $\sigma_i$ defined by}
%\begin{equation*}
%\sigma_{i}(t)=\int_0^t\left|\frac{df_i(x_i(t))}{dt}\right|
%\ dt\ ,
%\end{equation*}
%
%this formula
%is rewritten as
%\begin{equation}\label{eq:geodsigma1}
%\sum_{i=1}^n\int_{\sigma_i(s)}^{\sigma_i(t)}\frac{(-1)^iG(f_i)A(f_i)}
%{\sqrt{-\prod_{k=1}^{n-1}(f_i-b_k)
%\cdot\prod_{k=0}^n(f_i-a_k)}}\ 
%\ d\sigma_i=0\ .
%$\end{equation}
%
%Here, $f_i$ is regarded as a function of $\sigma_i$, i.e.,
%putting $\phi_i(t)
%=a_i+|t|$ for $|t|\le a_{i-1}-a_i$ and extending it to $\R$
%as a periodic function with the period $2(a_{i-1}-a_i)$, 
%we have
%\begin{equation*}
%f_i=
%\phi_i(\sigma_i+\epsilon_i(f_i(x_i(0))-a_i))\ ,
%\end{equation*}
%where $\epsilon_i=\pm 1$ is the sign of $df_i(x_i(t))/dt$ at 
%$t=0$.
Also, integrating $dt=2\omega_n=\sum_i(b_{in}/\xi_i)dx_i$, we have 
\begin{equation}\label{eq:geodlength}
\sum_{i=1}^n\int_{s}^{t}
\frac{(-1)^{i+1}\,\tilde G(f_i)}
{\sqrt{(-1)^{i-1}\prod_{k=1}^{n-1}(f_i-b_k)
}}\ 
\left|\frac{dx_i(t)}{dt}\right|
\ dt=t-s\ ,
\end{equation}
where $\tilde G(\lambda)$ is any monic polynomial in 
$\lambda$ of degree $n-1$.

Finally, let us illustrate the behavior of each coordinate function $x_i(t)$ along a geodesic $\gamma(t)=(x_1(t),\dots,x_n(t))$. Since
\begin{equation}\label{xprime}
x_i'(t)=\frac{\partial E}{\partial \xi_i}=\frac{\pm\sqrt{(-1)^{i-1}\prod_{k=1}^{n-1}(f_i(x_i)-b_k)}}{(-1)^{n-i}\prod_{l\ne i}(f_l(x_l)-f_i(x_i))}\,,
\end{equation}
and since
\begin{equation*}
f_i-b_{i+1}\ge f_i-f_{i+1}\ge 0,\quad
b_{i-2}-f_i\ge f_{i-1}-f_i\ge 0,
\end{equation*}
we have
\begin{equation*}
\left|x_i'(t)\right|\ge c\sqrt{(f_i(x_i))-b_i)(b_{i-1}-f_i(x_i))}
\end{equation*}
for some constant $c>0$ (at least, outside the branch locus of the covering $R\to M$). 
Therefore, if $a_i^+<f_i(x_i(t))<a_{i-1}^-$, then it reaches the boundary at a finite time $t=t_0$. And if $f_i(x_i(t_0))$
is equal to $b_i(\ne a_i)$ or $b_{i-1}(\ne a_{i-1})$, then $x'_i(t_0)=0$, $x''_i(t_0)\ne 0$, 
and if $f_i(x_i(t_0))$ is equal to $a_i(\ne b_i)$
or $a_{i-1} (\ne b_{i-1})$, then $x'_i(t_0)\ne 0$;
and $(d/dt)f_i(x_i(t))$ changes the sign when
$t$ passes $t_0$ in each case. Thus, if $a_i^-<a_i^+<a_{i-1}^-<a_{i-1}^+$, then $x_i(t)$ oscillates on $L_i$ 
if $L_i$ is an interval, or $x_i(t)$ moves monotonously if $L_i$ is the whole circle, and
the function $f_i(x_i(t))$ oscillates on the 
interval $[a_i^+,a_{i-1}^-]$.

If $a_{i+1}^+<a_i^-=a_i^+<a_{i-1}^-$ and if $x_i(t)$ and $x_{i+1}(t)$ satisfy 
\begin{equation*}
a_{i+1}^+<f_{i+1}(x_{i+1}(t))<a_i^-=a_i^+<f_i(x_i(t))<a_{i-1}^-,\quad \frac{d}{dt}f_i(x_i(t))<0
\end{equation*}
at some $t$, then $x_i(t)$ reaches the boundary point $\in f_i^{-1}(a_i^+)$ at a finite time $t_0$ by the same reason as above. Moreover, at that time, 
\begin{equation*}
f_{i+1}(x_{i+1}(t_0))=a_i^-=a_i^+=f_i(x_i(t_0))\,,
\end{equation*}
and this point $\gamma(t_0)$ is a branch point of the covering $R\to M$. In fact, observe the formula~\eqref{geodeqintf} for $G(\lambda)=\prod_{k\ne i}(\lambda-b_k)$ and take the limit $t\to t_0-0$ there. Since the $i$-th summand tends to $\infty$, if other $b_j$ and $a_k$ are all distinct, then the $(i+1)$-st summand must tend to
$\infty$ and thus $f_{i+1}(x_{i+1}(t_0))=a_i^\pm$.
In this case the geodesic passes through a point on $J_i$ and intersects $N_i$ transversally at
the point. For the general case we take a sequence of geodesics satisfying the above condition and obtain the same result.
Since $\gamma(t_0)$ is a branch point, there are two possible
ways of description for $x_i(t)$ after passing the time
$t_0$; but anyway, the function $f_i(x_i(t))$ turns the direction at 
$t=t_0$ and $(d/dt)(f_i(x_i(t))>0$ for $t>t_0$ near $t_0$. 
%The case where $a_i^+<a_{i-1}^-=a_{i-1}^+<a_{i-2}^-$ is similar.

If $a_i^+=a_{i-1}^-$, then there are several possible cases. If $a_i<b_i=b_{i-1}<a_{i-1}$, then $x_i(t)$ is constant along the geodesic. This case will be investigated in detail
in \S7. If $a_i=b_{i-1}$ (resp. if $b_i=a_{i-1}$), then the function $x_i(t)$ is again
constant, but in this case the geodesic is totally
contained in the totally geodesic submanifold $N_i$ (resp. $N_{i-1}$), and such type of geodesic is
not considered in this paper.

\section{A monotonicity condition for Liouville manifolds}
In this section we first introduce a monotonicity 
condition on
the positive function $A(\lambda)$, under which the structures 
of the conjugate loci on the corresponding Liouville manifolds 
become simple:
\begin{equation}\label{newcond}
\begin{gathered}
\text{The function }\tilde A(\lambda)=(\lambda-a_n)A(\lambda)
\text{ satisfies }\\
(-1)^{k}\tilde A^{(k)}(\lambda)>0\quad \text{on }
[a_n,a_0]\qquad (2\le k\le n),
\end{gathered}
\end{equation}
where $\tilde A^{(k)}(\lambda)$ denotes the $k$-th derivative of $\tilde A(\lambda)$ in $\lambda$. 
The following proposition indicates that this condition 
is stronger than the condition (4.1) in \cite{IK2},
which is:
\begin{equation}\label{eq:cut41}
(-1)^kA^{(k)}(\lambda)<0\quad \text{on }
[a_n,a_0]\quad (1\le k\le n-1)\ .
\end{equation}

\begin{prop}\label{prop:newcond}
If a positive function $A(\lambda)$ on $[a_n,a_0]$ satisfies
the condition (\ref{newcond}), then it also satisfies the
condition (\ref{eq:cut41}).
\end{prop}
\begin{proof}
Since $A(\lambda)$ is described in the form
\begin{equation*}
A(\lambda)=\frac{\tilde A(\lambda)-\tilde A(a_n)}{\lambda-a_n}
=\int_0^1\tilde A'(t\lambda+(1-t)a_n)\,dt\ ,
\end{equation*}
we have
\begin{equation*}
A^{(k)}(\lambda)
=\int_0^1t^k\,\tilde A^{(k+1)}(t\lambda+(1-t)a_n)\,dt
\qquad (1\le k\le n-1)\ .
\end{equation*}
The lemma immediately follows from this formula.
\end{proof}
\begin{remark}
If a positive function $A(\lambda)$ satisfies the condition
(\ref{newcond}), then $\tilde A'(\lambda)$ is positive on $[a_n,a_0]$.
In fact, $\tilde A'(\lambda)=A(\lambda)+(\lambda-a_n)A'(\lambda)$
and $A'(\lambda)>0$ by Proposition \ref{prop:newcond}.
\end{remark}
It is easily seen that $A(\lambda)=\sqrt{\lambda}$, i.e.,
the case of  the ellipsoid $\sum_iu_i^2/a_i=1$, satisfies the 
condition (\ref{newcond}).
{\it From now on, we shall always assume that the condition 
(\ref{newcond}) is satisfied.} 
Added to Proposition 4.1 in \cite{IK2},
we shall prove a similar proposition below. To do so, we need two lemmas; the first one
being the same as \cite[Lemmas 4.2]{IK2},
%and the second is similar to \cite[Lemma 4.3]{IK2}
we shall omit the proof. For the two lemmas we assume $b_1,\dots,b_{n-1}$ and 
$a_0,\dots,a_n$ are all distinct.
\begin{lemma}\label{lemma:flat}
\begin{equation*}
\sum_{i=1}^n\int_{a_i^+}^{a_{i-1}^-}
\frac{(-1)^{i}G(\lambda)\ d\lambda}
{\sqrt{-\prod_{k=1}^{n-1}(\lambda-b_k)
\cdot\prod_{k=0}^n(\lambda-a_k)}}
=0
\end{equation*}
for any polynomial $G(\lambda)$ of degree $\le n-2$.
\end{lemma}
\begin{lemma}\label{lemma:cond}
Let $J$ be any subset of $\{1,\dots,n-1\}$, and
let $B(\lambda)$ be the function defined by
\begin{equation}\label{eq:B}
\frac{A(\lambda)(\lambda-a_n)}{\prod_{j\in J}(\lambda-b_j)}=
\sum_{j\in J} \frac{e_j}{\lambda-b_j}+B(\lambda),
\quad e_j=\frac{A(b_j)(b_j-a_n)}{\prod_{\substack{l\in J\\
l\ne j}}(b_j-b_l)}.
\end{equation}
Then:
\begin{equation}\label{Brep}
B(\lambda)=\int_{D_k}\tilde A^{(k)}\left(\left(1-\sum_{l=1}^ks_l\right)\lambda+s_kb_{i_k}+\dots+s_1b_{i_1}\right)
ds_1\dots ds_k\,,
\end{equation}
where $J=\{i_1,\dots,i_k\}$ and
\begin{equation*}
D_k=\{(s_1,\dots,s_k)\in\R^k\,|\,s_i\ge 0\,(1\le i\le k),\,\sum_{i=1}^ks_i\le 1\}\,.
\end{equation*}
In particular, $B(\lambda)$ satisfies $(-1)^{\# J}B(\lambda)>0$ for $a_n\le \lambda\le a_0$ if $\# J\ge 2$.
\end{lemma}
\begin{proof}
We prove this by induction in $k=\#J$. When $k=0$, the assertion is trivial. Let $k\ge 0$ and assume
that the assertion is true for $J$ with $\#J\le k$. Suppose $J=\{i_1,\dots,i_{k+1}\}$ and put 
$J_0=J-\{i_{k+1}\}$.
Define $B_0(\lambda)$ as the function $B(\lambda)$ in the formula~\eqref{eq:B} for $J_0$.
By the induction assumption we have the formula~\eqref{Brep} for $B_0$.

By the defining formula~\eqref{eq:B}, the functions $B(\lambda)$ for $J$ and $B_0(\lambda)$ for $J_0$
are related as
\begin{equation*}
B(\lambda)=\frac{B_0(\lambda)-B_0(b_{i_{k+1}})}{\lambda-b_{i_{k+1}}}
=\int_0^1B_0'(t\lambda+(1-t)b_{i_{k+1}})\,dt\,.
\end{equation*}
By the induction assumption the right-hand side is equal to
\begin{equation*}\scriptstyle
\int_0^1\int_{D_k}\left(1-\sum_{l=1}^ks_l\right)\tilde A^{(k+1)}\left(\left(1-\sum_{l=1}^ks_l\right)
(t\lambda+(1-t)b_{i_{k+1}})+\sum_{l=1}^ks_lb_{i_l}\right)\,ds_1\dots ds_kdt
\end{equation*}
Therefore, changing the variable $t\to s_{k+1}=(1-\sum_{l=1}^ks_l)(1-t)$, we obtain the formula~\eqref{Brep}
for $J$.
\end{proof}

\begin{prop}\label{prop:newineq}
If $b_1,\dots,b_{n-1}$ and $a_0,\dots,a_n$ are all
distinct, then the following inequalities hold:
\begin{enumerate}
\item \begin{equation*}
\sum_{l=1}^n\int_{a_l^+}^{a_{l-1}^-}
\frac{(-1)^{n-l+\# I}A(\lambda)\,(\lambda-a_n)
\prod_{j\in I}(\lambda-b_j)
}
{\sqrt{-\prod_{k=1}^{n-1}(\lambda-b_k)
\cdot\prod_{k=0}^n(\lambda-a_k)}}\ 
d\lambda>0,
\end{equation*}
where $I$ is any (possibly empty)
subset of
$\{1,\dots, n-1\}$ such that $\# I\le n-3$;
\item \begin{equation*}
\frac{\partial }{\partial b_i}
\sum_{l=1}^n\int_{a_l^+}^{a_{l-1}^-}
\frac{(-1)^{l}A(\lambda)\,(\lambda-a_n)\,G(\lambda)\ d\lambda}
{\sqrt{-\prod_{k=1}^{n-1}(\lambda-b_k)
\cdot\prod_{k=0}^n(\lambda-a_k)}}
\end{equation*}
is negative for $G(\lambda)=\prod_{k\ne i}
(\lambda-b_k)$ and is positive for $G=
\prod_{k\ne i,j}(\lambda-b_k)$, $(j\ne i)$\ .
\item \begin{equation*}
\frac{\partial^2 }{\partial b_i^2}
\sum_{l=1}^n\int_{a_l^+}^{a_{l-1}^-}
\frac{(-1)^{l}A(\lambda)\,(\lambda-a_n)\,G(\lambda)\ d\lambda}
{\sqrt{-\prod_{k=1}^{n-1}(\lambda-b_k)
\cdot\prod_{k=0}^n(\lambda-a_k)}}
\end{equation*}
is positive for $G(\lambda)=\prod_{k\ne i}
(\lambda-b_k)$\ .
\end{enumerate}
\end{prop}
\begin{proof}
Put $J=\{1,\dots,n-1\}-I$, $\tilde A(\lambda)=
(\lambda-a_n)A(\lambda)$ and
define $B(\lambda)$ by
\begin{equation}\label{eq:interp}
\frac{\tilde A(\lambda)}{\prod_{j\in J}(\lambda-b_j)}=
\sum_{j\in J}\frac1{\lambda-b_j}\frac{\tilde A(b_j)}
{\prod_{\substack{k\in J\\k\ne j}}(b_j-b_k)}
+B(\lambda)\ .
\end{equation}
Then by Lemma~\ref{lemma:cond}  we have
$(-1)^{\# J}B(\lambda)>0$ on the interval $[a_n,a_0]$.
Since the sum in (1) is equal to
\begin{equation}\label{arrangedform}
\sum_{l=1}^n\int_{a_l^+}^{a_{l-1}^-}
\frac{(-1)^{l-1+\# J}B(\lambda)
\prod_{k=1}^{n-1}(\lambda-b_k)
}
{\sqrt{-\prod_{k=1}^{n-1}(\lambda-b_k)
\cdot\prod_{k=0}^n(\lambda-a_k)}}\ d\lambda
\end{equation}
by Lemma~\ref{lemma:flat}, and since
$(-1)^{l-1}\prod_{k=1}^{n-1}(\lambda-b_k)$ is 
positive on every interval $(a_l^+,a_{l-1}^-)$,
we have the inequality (1).

To prove (2) for $G(\lambda)=\prod_{k\ne i,j}(\lambda-b_k)$, we use the formula (\ref{eq:interp}) with $J=
\{i,j\}$. In this case, 
\begin{equation*}
B(\lambda)=\int_0^1\!\!\int_0^{1-t}\tilde A''((1-t-s)\lambda+tb_i+
sb_j)\,dsdt\ ,
\end{equation*}
and the formula in (2) is written as
\begin{equation}\label{arrangedform2}
\frac{\partial }{\partial b_i}
\sum_{l=1}^n\int_{a_l^+}^{a_{l-1}^-}
\frac{(-1)^{l}B(\lambda)\,\prod_{k=1}^{n-1}(\lambda-b_k)}
{\sqrt{-\prod_{k=1}^{n-1}(\lambda-b_k)
\cdot\prod_{k=0}^n(\lambda-a_k)}}\ d\lambda\ .
\end{equation}
Then, in the same way as the proof of Proposition 4.1 (2) in \cite{IK2},
we see that
the above formula is equal to
\begin{equation}\label{firstdiff}
\frac12\sum_{l=1}^n\int_{a_l^+}^{a_{l-1}^-}
\frac{(-1)^{l}\left(\frac{\partial }{\partial b_i}B(\lambda)
\right)\prod_{k=1}^{n-1}(\lambda-b_k)}
{\sqrt{-\prod_{k=1}^{n-1}(\lambda-b_k)
\cdot\prod_{k=0}^n(\lambda-a_k)}}\ d\lambda\ ,
\end{equation}
which is positive, since $\frac{\partial }
{\partial b_i}B(\lambda)<0$.

In the case where $G(\lambda)=\prod_{k\ne i}(\lambda-b_k)$,
we also have the same formula as above with
\begin{equation}\label{fnb}
B(\lambda)=\int_0^1\tilde A'(t\lambda+(1-t)b_i)dt
\end{equation}
Since $\frac{\partial}{\partial b_i}B(\lambda)>0$ in this case,
the assertion follows.

(3) Differentiating the formula \eqref{firstdiff}
by $b_i$ under the equality \eqref{fnb}, we have
\begin{equation}\label{seconddiff}
\begin{gathered}
\frac12\sum_{l=1}^n\int_{a_l^+}^{a_{l-1}^-}
\frac{(-1)^{l}\left(\frac{\partial^2 }{\partial b_i^2}B(\lambda)
\right)\prod_{k=1}^{n-1}(\lambda-b_k)}
{\sqrt{-\prod_{k=1}^{n-1}(\lambda-b_k)
\cdot\prod_{k=0}^n(\lambda-a_k)}}\ d\lambda\\
-\frac14\sum_{l=1}^n\int_{a_l^+}^{a_{l-1}^-}
\frac{(-1)^{l}\left(\frac{\partial }{\partial b_i}B(\lambda)
\right)\prod_{k\ne i}(\lambda-b_k)}
{\sqrt{-\prod_{k=1}^{n-1}(\lambda-b_k)
\cdot\prod_{k=0}^n(\lambda-a_k)}}\ d\lambda\  .
\end{gathered}\end{equation}
By Lemma~\ref{lemma:flat} in \cite{IK2} the second line of this
formula is equal to 
\begin{equation}\label{arrangedform3}
-\frac14\sum_{l=1}^n\int_{a_l^+}^{a_{l-1}^-}
\frac{(-1)^{l}{\tilde B}(\lambda)
\prod_{k=1}^{n-1}(\lambda-b_k)}
{\sqrt{-\prod_{k=1}^{n-1}(\lambda-b_k)
\cdot\prod_{k=0}^n(\lambda-a_k)}}\ d\lambda\  ,
\end{equation}
where 
\begin{equation*}
\tilde B(\lambda)=\frac{\frac{\partial }{\partial b_i}B(\lambda)-\frac12 \tilde A''(b_i)}
{\lambda-b_i}=\frac12\frac{\partial^2 }{\partial b_i^2}B(\lambda)\ .
\end{equation*}
Note that $(\partial/\partial b_i)B(\lambda)|_{\lambda=b_i}=(1/2)\tilde A''(b_i)$.

Therefore the formula \eqref{seconddiff} is equal
to
\begin{equation}\label{secondderB}
\frac38\sum_{l=1}^n\int_{a_l^+}^{a_{l-1}^-}
\frac{(-1)^{l}\left(\frac{\partial^2 }{\partial b_i^2}B(\lambda)
\right)\prod_{k=1}^{n-1}(\lambda-b_k)}
{\sqrt{-\prod_{k=1}^{n-1}(\lambda-b_k)
\cdot\prod_{k=0}^n(\lambda-a_k)}}\ d\lambda\ .
\end{equation}
Since 
\begin{equation}
\frac{\partial^2 }{\partial b_i^2}B(\lambda)=
\int_0^1(1-t)^2\tilde A'''(t\lambda+(1-t)b_i)dt
\end{equation}
is negative, the assertion follows.
\end{proof}

In the later applications we also need certain limit 
cases of the above proposition, which may be stated as follows.
\begin{prop}\label{prop:limitineq}
Let $b^k=(b_1^k,\dots,b_{n-1}^k)$ $(k=1,2,\dots)$ be a sequence such that
\begin{equation*}
a_{i+1}<b_i^k<a_{i-1},\, b_i^k\ne a_i,\, b_i^k<b_{i-1}^k \quad \text{for any }k,i,
\end{equation*}
and such that the ordering of $a_i$ and $b_i^k$ does not change when $k$ varies for each $i$.
Suppose that $b^k$ converges to 
$b^\infty=(b_1,\dots,b_{n-1})$ as $k\to\infty$. Then, when $k\to \infty$, each formula in
$(1)$, $(2)$, $(3)$ in Proposition~\ref{prop:newineq}
for $b^k$ converges to a nonzero value. Namely, those inequalities are still valid in the limit case.
\end{prop}
\begin{proof}
We first consider the case (1) in Proposition~\ref{prop:newineq}. Let us observe the formula~\eqref{arrangedform} for $b^k$ and take the limit $k\to \infty$. 
We regard the sum of the integrals as the integral over $[a_n,a_0]$ of the single function $E^k(\lambda)$,
where
\begin{equation*}
E^k(\lambda)=\begin{cases}
\frac{(-1)^{l-1+\# J}B(\lambda)
\prod_{i=1}^{n-1}(\lambda-b_i^k)}{\sqrt{-\prod_{i=1}^{n-1}(\lambda-b_i^k)
\cdot\prod_{i=0}^n(\lambda-a_i)}}\quad &(\lambda\in[a_l^+,a_{l-1}^-])\\
0\quad &(\lambda\not\in \cup_{i=1}^n[a_i^+,a_{i-1}^-])
\end{cases}.
\end{equation*}
In view of the formula~\eqref{Brep} we see that there is a constant $c$ which does not depend on $k$ such that
\begin{equation*}
|E^k(\lambda)|\le \frac{c}{\sqrt{\prod_{i=0}^n|\lambda-a_i|}}\qquad (a_n\le \lambda\le a_0)
\end{equation*}
for any $k$. Therefore, by Lebesgue's convergence theorem, we have
\begin{equation*}
\lim_{k\to\infty}\int_{a_n}^{a_0}E^k(\lambda)d\lambda=\int_{a_n}^{a_0}E^\infty(\lambda)d\lambda.
\quad E^\infty(\lambda)=\lim_{k\to\infty}E^k(\lambda),
\end{equation*}

Since there is at least one index $i$ such that $a_i^+= a_i$ and $a_{i-1}^-=a_{i-1}$ for each $k$ 
and since this index $i$ does not depend on $k$ by the assumption, it follows that $a_i^+= a_i$ and
$a_{i-1}^-= a_{i-1}$ for $k=\infty$. Therefore $E^\infty(\lambda)$ is positive on the open interval $(a_i,a_{i-1})$
and nonnegative on the whole interval $[a_n,a_0]$. Thus the assertion follows.

For (2) and (3) we use \eqref{firstdiff} and \eqref{secondderB} instead of \eqref{arrangedform}.
Since the proof goes in completely the same way as above, we omit the detail.
\end{proof}

\section{Zeros of Jacobi fields}
Let $\gamma(t)=(x_1(t),\dots,x_n(t))$ be a geodesic which is not totally contained in 
the totally geodesic submanifolds $N_i$ $(1\le i\le n-1)$.
Let us denote by $\sigma_i(t)$
the total variation of $f_i(x_i(t))$:
\begin{equation*}
\sigma_i(t)=\int_0^t\left|\frac{df_i(x_i(s))}{ds}
\right|ds\qquad (1\le i\le n)\ .
\end{equation*}
When $a_i^+< a_{i-1}^-$, this function is strictly increasing (cf. \S3), and we then define the time $t=t_i$ by the equality 
\begin{equation}\label{def:ti}
\sigma_i(t_i)=2(a_{i-1}^--a_i^+)\,,
\end{equation}
which represents a half of the period in some sense.
Note that $t_n$ is the same 
one as $t_0$ defined in \cite[\S6]{IK2}. 
Note also that, in view of \eqref{eq:fdiff},
the following equalities hold:
\begin{equation}\label{basiceq}
\begin{gathered}
\int_0^{t_i}\frac{(-1)^iG(f_i)}
{\sqrt{(-1)^{i-1}\prod_{k=1}^{n-1}(f_i-b_k)
}}\ 
\left|\frac{dx_i(t)}{dt}\right|
\ dt\\
=\frac12\int_0^{t_i}\frac{(-1)^iG(f_i)A(f_i)}
{\sqrt{-\prod_{k=1}^{n-1}(f_i-b_k)
\cdot\prod_{k=0}^n(f_i-a_k)}}\ 
\left|\frac{df_i(x_i(t))}{dt}\right|
\ dt\\
=\int_{a_i^+}^{a_{i-1}^-}
\frac{(-1)^iG(\lambda)
A(\lambda)\ d\lambda}
{\sqrt{-\prod_{k=1}^{n-1}(\lambda-b_k)
\cdot\prod_{k=0}^n(\lambda-a_k)}}\ .
\end{gathered}
\end{equation}
Those equalities will be frequently used below.

In the rest of this section we shall assume that
the corresponding $n-1$ constants $b_i$ 
and the $n+1$ constants $a_j$ are all distinct
unless otherwise stated.
We have already seen in \cite[Proposition 6.5]{IK2} that $t_n<t_i$ for any $i\le n-1$.
Here one can obtain a stronger result.
\begin{prop}\label{prop:order1}
$t_{n}<t_{n-1}<\dots<t_1$.
\end{prop}
\begin{proof}
Fix $k$ such that $2\le k\le n-1$ and assume that $t_i\le t_{k}$ for some $i\le k-1$. Put
\begin{equation*}
I=\{i\ |\ 1\le i\le n,\ t_i\le t_k,\ i\ne k\}
\ .
\end{equation*} 
Let $J$ be the set of $j$ such that
$1\le j\le n-1$ and either $j\in I$ and
$j+1\in I$, or
$j\not\in I$ and $j+1\not\in I$.
Since there is some $i\in I$ such that $i<k$ by
the assumption, and since $n\in I$ as remarked above and $k\not\in I$, it follows that $\# J\le n-3$.

We then consider the equality (the geodesic equation)
\begin{equation}\label{eq:T}
\begin{aligned}
&\sum_{l=1}^{n}\int_{t_l}^
{t_k}\frac{(-1)^lG(f_l)}
{\sqrt{(-1)^{l-1}\prod_{k=1}^{n-1}(f_l-b_k)
}}\left|\frac{dx_l(t)}{dt}\right|
\ dt\\
&+\sum_{l=1}^n\int_{a_l^+}^{a_{l-1}^-}
\frac{(-1)^lG(\lambda)
A(\lambda)\ d\lambda}
{\sqrt{-\prod_{k=1}^{n-1}(\lambda-b_k)
\cdot\prod_{k=0}^n(\lambda-a_k)}}=0\ ,
\end{aligned}
\end{equation}
where 
$G(\lambda)=
(\lambda-a_n)\prod_{j\in J}(\lambda-b_j)$\ .
Since the sign of $(-1)^lG(f_l)$ are the same for
any $l\in I$, and since $n\in I$, it follows that
\begin{equation*}
(-1)^{n-\#J+l} G(f_l)
\begin{cases} \ge 0\quad (l\in I)\\
\le 0\quad (l\not\in I)
\end{cases}.
\end{equation*}
 Also, we have
\begin{equation*}
t_k\ge t_l\quad
(l\in I),\quad t_k\le t_l\quad
(l\not\in I)\ .
\end{equation*}
Therefore the sign of the first line of the
formula (\ref{eq:T}) is $(-1)^{n-\# J}$.
On the other hand, the second line of
(\ref{eq:T}) is nonzero and its sign is, by Proposition 
\ref{prop:newineq} (1), equal to $(-1)^{n-\# J}$, 
which is a contradiction. Therefore we have 
$t_k<t_i$ for any $1\le i\le k-1$, and the
proposition thus follows.
\end{proof}

Let $H_i$ ($1\le i\le n-1$) denote the first
integral of the geodesic flow whose value is
expressed by $b_i$, i.e., $H_i$ are functions
on the unit cotangent bundle $U^*M$ defined by
the following identity in $\lambda$;
\begin{gather*}
\sum_{j=1}^{n-1}\left(
\prod_{\substack{1\le k\le n-1\\ k\ne j}}
(\lambda-a_k)\right)\,F_j(\mu)-\prod_{k=1}^{n-1}(\lambda-a_k)=-\prod_{i=1}^{n-1}(\lambda-H_i(\mu))\,,\\
H_1(\mu)\ge H_2(\mu)\ge \cdots \ge H_{n-1}(\mu)\,,\quad \mu\in U^*M.
\end{gather*}
We extend $H_i$ to $T^*M-\{0$-section\} as a function of degree $0$, i.e.,
\begin{equation*}
H_i(t\lambda)=H_i(\lambda),\qquad t>0, \lambda\in T^*M, \lambda\ne 0\,.
\end{equation*}
Then the $\pi_*$-image of the vector $X_{H_i}$ at $\flat(\dot\gamma(t))$ is perpendicular to
$\dot\gamma(t)$. 

In \cite[Proposition 5.1]{IK2} we proved that
the Jacobi fields on the manifolds we are considering possess a remarkable property,
which may be stated as follows.
\begin{prop}
There are smooth vector fields $V_i(t)$ $(1\le i\le n-1)$ 
along the geodesic $\gamma(t)$ satisfying
\begin{equation*}%\label{parallelvf}
\pi_*\left((X_{H_i})_{\flat(\dot\gamma(t))}\right)=h_i(t)V_i(t)\,,\quad |V_i(t)|=1
\end{equation*}
for some functions $h_i(t)$ and they have the 
following properties:
\begin{enumerate}
\item Each $V_i(t)$ is parallel along
$\gamma(t)$.
\item $V_1(t),\dots,V_{n-1}(t)$ are mutually orthogonal for any $t\in \R$.
\item Any Jacobi field $Z(t)$ satisfying $Z(0), Z'(0)\in \R V_i(0)$ is of the form $z(t)V_i(t)$ with some function $z(t)$ for any $t\in\R$ and any $i$.
\end{enumerate}
\end{prop}
We prove here the following proposition, which will be necessary in later sections.
\begin{prop}\label{prop:dualtob}
The one-form
\begin{equation*}
\tilde \omega_i=\sum_{k=1}^n\frac{\epsilon_k(-1)^kG_i(f_k)\ dx_k}
{\sqrt{(-1)^{k-1}\prod_{l=1}^{n-1}(f_k-b_l)}},\quad G_i(\lambda)=\prod_{1\le l\le n-1\atop l\ne i}
(\lambda-b_l),
\end{equation*}
satisfies 
\begin{equation*}
\tilde \omega_i(V_k(t))=0\quad (k\ne i),\quad \tilde \omega(\dot\gamma(t))=0
\end{equation*}
at $\gamma(t)$ for any $t\in\R$ such that $(f_{i+1}(x_{i+1}(t))-b_i)(f_i(x_i(t))-b_i)\ne 0$ $(1\le  i\le n-1)$. Here $\epsilon_k=$ sign of $\xi_k=$ sign of $x_k'(t)$. At $t\in\R$ with $f_{i+1}(x_{i+1}(t))=b_i$ (resp. $f_i(x_i(t))=b_i$)
the one-form $dx_{i+1}$ (resp. $dx_i$) has the same property.
\end{prop}
\begin{proof}
By the identity
\begin{equation}\label{idHi}
\prod_{l=1}^{n-1}(f_k(x_k)-H_l)\cdot 2E=(-1)^{k+1}\xi_k^2
\end{equation}
one obtains
\begin{equation*}
\sum_{i=1}^{n-1}\prod_{l\ne i}(f_k-b_l)\cdot\pi_*(X_{H_i})-2\prod_{l=1}^{n-1}(f_k-b_l)\cdot
\pi_*(X_E)=(-1)^k2\xi_k\frac{\partial}{\partial x_k}
\end{equation*}
for $(x,\xi)\in U^*M$ at which $H_l=b_l$ (and $2E=1$).
Then, taking the dual one-forms $\eta_1,\dots,\eta_n$ of $\pi_*(X_{H_1}),\dots,\pi_*(X_{H_{n-1}}),\pi_*(X_E)$, we have
\begin{equation*}
\eta_i=\sum_{k=1}^n\frac{(-1)^k\prod_{l\ne i}(f_k-b_l)}{2\xi_k}\,dx_k\quad(1\le i\le n-1)\,.
\end{equation*}
Then, taking \eqref{idHi} into account, we have the proposition for the points where $\pi_*(X_{H_i})\ne 0$
for any $i$. For the point where $\pi_*(X_{H_i})=0$, i.e., $f_{i+1}(x_{i+1})=b_i$ or $f_i(x_i)=b_i$, we can take the limit:
\begin{align*}
&\lim_{f_{i+1}(x_{i+1})\to b_i}\epsilon_{i+1}\sqrt{b_i-f_{i+1}(x_{i+1})}\,\tilde \omega_i=\sqrt{G_i(f_{i+1})}
\,dx_{i+1}\,,\\
&\lim_{f_i(x_i)\to b_i}\epsilon_i\sqrt{f_i(x_i)-b_i}\,\tilde\omega_i=\sqrt{G_i(f_i)}\,dx_i\,.
\end{align*}
Thus the proposition follows.
\end{proof}
Let us define the Jacobi field $Y_i(t)$ by the initial condition:
\begin{equation*}
Y_i(0)=0, \quad Y'_i(0)=V_i(0)\quad (1\le i\le n-1).
\end{equation*}
Then $Y_i(t)$ is of the form $y_i(t)V_i(t)$ for some function $y_i(t)$.
Let $t=r_i$ be the first zero of $Y_i(t)$ for $t>0$. We have 
already seen that $r_i\ge t_n$ for any $i$ 
(\cite[Proposition 5.3]{IK2}). Moreover,
let $S_i$ be the discrete subset of $\R$ such that
\begin{equation*}
t\in S_i\Longleftrightarrow
\begin{cases}
f_i(x_i(t))=b_i\quad \text{if } b_i=a_i^+\\
f_{i+1}(x_{i+1}(t))=b_i\quad \text{if } b_i=a_i^-
\end{cases}
\end{equation*}
as given in \cite[\S5]{IK2}. Let $s_i^1$ and $s_i^2$ be
the first and the second positive time in $S_i$
respectively. We then have, by the definition and
Proposition~5.1 in \cite{IK2},
\begin{gather}\label{primeineq}
s_i^1<r_i,\, t_i<s_i^2\quad (b_i=a_i^+)\ ,\quad
s_i^1<r_i,\, t_{i+1}<s_i^2\quad (b_i=a_i^-)\quad
\text{if}\quad 0\not\in S_i\\
\label{primeineq2}r_i=s_i^1=t_i\quad (b_i=a_i^+)\ ,\quad
r_i=s_i^1=t_{i+1}\quad (b_i=a_i^-)\quad
\text{if}\quad 0\in S_i\ .
\end{gather}

Now we prove the following proposition about the ordering of $r_i$ and $t_j$.
\begin{prop}\label{prop:main}
If \,$0\not\in S_j$, then \,$t_{j+1}<r_j<t_j$ for each $1\le j\le n-1$.
\end{prop}
\begin{proof}
We shall first prove that $r_j<t_1$ for any $j$. Suppose that
$r_j\ge t_1$ for some $j$. Let $G(\lambda)=\prod_{k\ne j}
(\lambda-b_k)$ and observe the following formula 
(a part of geodesic equations):
\begin{equation}\label{eq:geod3}
\begin{aligned}
&t=\sum_{l=1}^{n}\int_{t_l}^{t
}\frac{(-1)^lG(f_l)\,(f_l-a_n)}
{\sqrt{(-1)^{l-1}\prod_{k=1}^{n-1}(f_l-b_k)
}}\left|\frac{dx_l(t)}{dt}\right|
\ dt\\
&+\sum_{l=1}^n\int_{a_l^+}^{a_{l-1}^-}
\frac{(-1)^lG(\lambda)\,(\lambda-a_n)
A(\lambda)\ d\lambda}
{\sqrt{-\prod_{k=1}^{n-1}(\lambda-b_k)
\cdot\prod_{k=0}^n(\lambda-a_k)}}\ .
\end{aligned}
\end{equation}
We differentiate the above formula
in terms of the deformation parameter defining 
the Jacobi field $cY_{j}$,
$c$ being $\pm$ (the norm of $\partial/\partial b_j$ at 
$\gamma(0)$), and put $t=r_j$. Then, we
claim that the resulting formula is:
\begin{equation}\label{eq:diff3}
\begin{aligned}
&\frac1{2}\sum_{l=1}^{n}
\int_{t_l}^{r_j}\frac{(-1)^lG(f_l)\,(f_l-a_n)}
{(f_l-b_j)\sqrt{(-1)^{l-1}\prod_{k=1}^{n-1}(f_l-b_k)}}
\left|\frac{dx_l(t)}{dt}\right|
\ dt\\
&+\frac{\partial}{\partial b_j}
\sum_{l=1}^n\int_{a_l^+}^{a_{l-1}^-}
\frac{(-1)^lG(\lambda)\,(\lambda-a_n)
A(\lambda)\ d\lambda}
{\sqrt{-\prod_{k=1}^{n-1}(\lambda-b_k)
\cdot\prod_{k=0}^n(\lambda-a_k)}}=0\ .
\end{aligned}
\end{equation}

To prove this we first assume that $b_j=a_j^+$.
Let 
\begin{equation*}
\gamma(t,u)=(x_1(t,u),\dots,x_n(t,u))\quad (|u|<\epsilon)
\end{equation*}
be a variation of the geodesic $\gamma(t)=\gamma(t,0)$ such that
\begin{equation*}
\frac{\partial \gamma}{\partial u}(t,0)=cY_j(t),
\end{equation*}
and that the value of each first integral $H_k$ ($k\ne j$) for the geodesics $t\to\gamma(t,u)$ remains to be $b_k$ (constant) for any $u$. (In this case,
$b_j$ is a function of $u$ such that $db_j/du=1$ at $u=0$.) For $l\ne j$ and for $t>t_l$
satisfying
\begin{equation*}
f_l(x_l(t,0))\ne a_l^+, a_{l-1}^-\,,
\end{equation*}
we define times $t^*$ and $\hat t$ such that
$t_l< t^*\le\hat t<t$ and that
\begin{gather*}
f_l(x_l(t^*,u)),\ f_l(x_l(\hat t,u))=\ a_l^+\text{ or } a_{l-1}^-\,,\\
a_l^+<f_l(x_l(s,u))<a_{l-1}^- \quad \text{for } s\in[\,t_l,t^*)\cup(\hat t,t\,]\,.
\end{gather*}
Then one obtains the following expression for sufficiently small $|u|$:
\begin{gather*}
\int_{t_l}^t\frac{(-1)^lG(f_l)\,(f_l-a_n)}
{\sqrt{(-1)^{l-1}\prod_{k=1}^{n-1}(f_l-b_k)
}}\left|\frac{\partial x_l(s,u)}{\partial s}\right|\ ds\\
=\frac{\epsilon}2\int_{f_l(x_l(t_l,u))}^{f_l(x_l(t^*,u))}\frac{(-1)^lG(\lambda)\,(\lambda-a_n)
A(\lambda)\ d\lambda}
{\sqrt{-\prod_{k=1}^{n-1}(\lambda-b_k)
\cdot\prod_{k=0}^n(\lambda-a_k)}}\\
+\frac{k}2\int_{a_l^+}^{a_{l-1}^-}
\frac{(-1)^lG(\lambda)\,(\lambda-a_n)
A(\lambda)\ d\lambda}
{\sqrt{-\prod_{k=1}^{n-1}(\lambda-b_k)
\cdot\prod_{k=0}^n(\lambda-a_k)}}\\
+\,\frac{\epsilon'}2\int_{f_l(x_l(\hat t,u))}^{f_l(x_l(t,u))}\frac{(-1)^lG(\lambda)\,(\lambda-a_n)
A(\lambda)\ d\lambda}
{\sqrt{-\prod_{k=1}^{n-1}(\lambda-b_k)
\cdot\prod_{k=0}^n(\lambda-a_k)}}\,,
\end{gather*}
where $\epsilon,\epsilon'\,(=\pm 1)$ are the sign of 
$\partial f_j(x_l(s,u))/\partial s$ at $s=t_l$ and $s=t$ respectively and $k$ is a certain nonnegative integer.

We differentiate the above formula in $u$ and put
$u=0$. 
Since $f_l(x_l(t_l,u))$, $f_l(x_l(t^*,u))$, and
$f_l(x_l(\hat t,u))$ do not depend on $u$,
the right-hand side becomes
\begin{equation}\label{dif}
\begin{gathered}
\frac{\epsilon}4\int_{f_l(x_l(t_l,u))}^{f_l(x_l(t^*,u))}\frac{(-1)^lG(\lambda)\,(\lambda-a_n)
A(\lambda)\ d\lambda}
{(\lambda-b_j)\sqrt{-\prod_{k=1}^{n-1}(\lambda-b_k)
\cdot\prod_{k=0}^n(\lambda-a_k)}}\\
+\frac{k}4\int_{a_l^+}^{a_{l-1}^-}
\frac{(-1)^lG(\lambda)\,(\lambda-a_n)
A(\lambda)\ d\lambda}
{(\lambda-b_j)\sqrt{-\prod_{k=1}^{n-1}(\lambda-b_k)
\cdot\prod_{k=0}^n(\lambda-a_k)}}\\
+\,\frac{\epsilon'}4\int_{f_l(x_l(\hat t,u))}^{f_l(x_l(t,u))}\frac{(-1)^lG(\lambda)\,(\lambda-a_n)
A(\lambda)\ d\lambda}
{(\lambda-b_j)\sqrt{-\prod_{k=1}^{n-1}(\lambda-b_k)
\cdot\prod_{k=0}^n(\lambda-a_k)}}\\
+\frac{\epsilon'}2\frac{(-1)^lG(f_l)\,(f_l-a_n)A(f_l)}
{\sqrt{-\prod_{k=1}^{n-1}(f_l-b_k)\cdot\prod_{k=0}^n(f_l-a_k)}}\,f_l'(x_l)\,dx_l(cY_j(t))\,.
\end{gathered}
\end{equation}
In the same way, it turns out that the sum of the first three lines of the formula
\eqref{dif} is equal to
\begin{equation}\label{dif2}
\frac12\int_{t_l}^{t}\frac{(-1)^lG(f_l)\,(f_l-a_n)}
{(f_l-b_j)\sqrt{(-1)^{l-1}\prod_{k=1}^{n-1}(f_l-b_k)
}}\left|\frac{d x_l(t)}{d t}\right|\ dt\,.
\end{equation}
Thus, if $f_l(x_l(r_j))\ne a_l^+, a_{l-1}^-$, then putting $t=r_j$ in \eqref{dif}, we have the
desired formula for $l$. If $f_l(x_l(r_j))= a_l^+, a_{l-1}^-$, then taking $t<r_j$ and taking the
limit $t\to r_j$, one obtains the same formula.

In case there are no such times $t^*$ and $\hat t$,
i.e., if 
\begin{equation*}
a_l^+<f_l(x_l(s,u))<a_{l-1}^- \quad \text{for any } s\in[\,t_l, t\,]\,,
\end{equation*}
then instead of \eqref{dif} one has, 
\begin{equation}\label{dif3}
\begin{gathered}
\frac{\epsilon}4\int_{f_l(x_l(t_l,u))}^{f_l(x_l(t,u))}\frac{(-1)^lG(\lambda)\,(\lambda-a_n)
A(\lambda)\ d\lambda}
{(\lambda-b_j)\sqrt{-\prod_{k=1}^{n-1}(\lambda-b_k)
\cdot\prod_{k=0}^n(\lambda-a_k)}}\\
+\frac{\epsilon'}2\frac{(-1)^lG(f_l)\,(f_l-a_n)A(f_l)}
{\sqrt{-\prod_{k=1}^{n-1}(f_l-b_k)\cdot\prod_{k=0}^n(\lambda-a_k)}}\,f_l'(x_l)\,dx_l(cY_j(t))\,,
\end{gathered}
\end{equation}
the first line of which is again equal to \eqref{dif2}. Thus we have the same result in this case.

Next, we consider the remaining term in 
\eqref{eq:geod3}:
\begin{equation}\label{eq:j}
\int_{t_j}^t\frac{(-1)^jG(f_j)\,(f_j-a_n)}
{\sqrt{(-1)^{j-1}\prod_{k=1}^{n-1}(f_j-b_k)
}}\left|\frac{\partial x_j(t,u)}{\partial t}\right|\ dt\,.
\end{equation}
Let us differentiate \eqref{eq:j} in $u$ at $u=0$ and put $t=r_j$.
When $t$ is close to $r_j$, the inequalities \eqref{primeineq} imply that $f_j(x_j(s))-b_j$
does not vanish on the interval $t_j\le s\le t$.
Thus the derivative of \eqref{eq:j} in $u$ at $u=0$ is described as \eqref{dif} with $k=0$ or as \eqref{dif3}.
Therefore, putting $t=r_j$, one obtains
\begin{equation*}
\frac12\int_{t_j}^{r_j}\frac{(-1)^jG(f_j)\,(f_j-a_n)}
{(f_j-b_j)\sqrt{(-1)^{j-1}\prod_{k=1}^{n-1}(f_j-b_k)
}}\left|\frac{d x_j(t)}{d t}\right|\ dt\,.
\end{equation*}
Hence the formula \eqref{eq:diff3} follows when 
$b_j=a_j^+$. The case where $b_j=a_{j}^-$ is similar and we omit the detail.

Let us come back to the situation before the claim. Since $r_j\ge t_1\ge t_l$, the first line of
the formula (\ref{eq:diff3}) is nonpositive. However, since
the second line is negative by Proposition \ref{prop:newineq} (2),
it is a contradiction. Thus $r_j<t_1$ for any $j$.

Now, we have proved that $t_n<r_j<t_1$. Assume that
$t_{m+1}\le r_j\le t_m$ for some 
$m\ne j$ and put
$G(\lambda)=\prod_{l\ne m,j}(\lambda-b_l)$. Then,
differentiating the formula
\begin{equation}
\begin{aligned}
&\sum_{l=1}^{n}\int_{t_l}^{t}\frac{(-1)^lG(f_l)\,(f_l-a_n)}
{\sqrt{(-)^{l-1}\prod_{k=1}^{n-1}(f_l-b_k)}}
\left|\frac{dx_l(t)}{dt}\right|
\ dt\\
&+\sum_{l=1}^n\int_{a_l^+}^{a_{l-1}^-}
\frac{(-1)^lG(\lambda)\,(\lambda-a_n)
A(\lambda)\ d\lambda}
{\sqrt{-\prod_{k=1}^{n-1}(\lambda-b_k)
\cdot\prod_{k=0}^n(\lambda-a_k)}}=0
\end{aligned}
\end{equation}
by $cY_{j}$ as above and putting $t=r_j$, we have the same
formula as (\ref{eq:diff3}) with $G(\lambda)=\prod_{l\ne m,j}
(\lambda-b_l)$. In this case, each summand of the first line of
(\ref{eq:diff3}) is nonnegative, whereas the second line is positive
by Proposition~\ref{prop:newineq}; again a contradiction.
Thus we have $t_{j+1}<r_j<t_j$ for any $j$.
\end{proof}
As an application of Proposition \ref{prop:main},
we shall show that the distribution of conjugate 
points have some curious asymptotic property. 
Let $\gamma(t)=(x_1(t),
\dots,x_n(t))$ be a geodesic such that the corresponding
$n-1$ values $b_i$ of the first integrals $H_i$
and the $n+1$ constants $a_j$ are all distinct.
Let $L\subset M$ be the $\pi$-image of the Lagrange
torus in $T^*M$ determined by $\cap_{i}
H_i^{-1}(b_i)$ and containing the geodesic orbit
$\{\flat(\dot\gamma(t))\}$, where $\pi:T^*M\to M$ is
the bundle projection. As stated in \S3, $L$ is 
diffeomorphic to the product
$L_1\times\dots\times L_n$, where each $L_j$ is either  the whole
circle $\R/\alpha_j\Z$ or an arc in it. Let $t=r_i^k$ be the
$k$-th zero of the Jacobi field $Y_i(t)$ $(1\le i\le n-1)$; 
$0<r_i^1<r_i^2<\cdots$.
\begin{thm}The sequence $\{\gamma(r_i^k)\}_{k\in\N}$ of conjugate points
of $\gamma(0)$ approaches the boundary $\partial L$ of $L$ 
when $k\to\infty$, 
i.e., for any $\epsilon >0$, there is $N>0$ such that the distance
of $\gamma(r_i^k)$ from $\partial L$ is less than $\epsilon$, if
$k\ge N$.
\end{thm}
\begin{proof}
We shall consider the case where $b_i=a_i^+$. The case where
$b_i=a_{i}^-$ will be similar. Let $\{s_i^k\}$, $0<s_i^1<s_i^2<\cdots$, be the set of times such that $f_i(x_i(s_i^k))=b_i$.
Then, by Corollary 5.2 in \cite{IK2} and Proposition \ref{prop:main},
we have $s_i^k<r_i^k<s_i^{k+1}<r_i^{k+1}$ and
\begin{equation*}
|x_i(r_i^k)-x_i(s_i^k)|>|x_i(r_i^{k+1})-x_i(s_i^{k+1})|\ .
\end{equation*}
Note that $\gamma(s_i^k)\in L_1\times\dots\times\partial L_i\times\dots\times L_n\subset\partial L$.

We shall show that 
\begin{equation*}
\lim_{k\to\infty} |x_i(r_i^k)-x_i(s_i^k)|=0\ ,
\end{equation*}
which will indicate the theorem. Assume that this is not the case.
Then, there is a subsequence $r_i^{k(l)}$ $(l=1,2,\dots)$ such that
\begin{equation*}
\lim_{l\to\infty} x_i(r_i^{k(l)})=a,\quad f_i(a)>b_i,\quad
\lim_{l\to\infty}\gamma(r_i^{k(l)})=p\in L\ .
\end{equation*}
Then, taking a subsequence if necessary, the sequence 
of geodesics 
\newline$\gamma(t+r_i^{k(l)})$ converges
to a geodesic $\tilde\gamma(t)=(\tilde x_1(t),\dots,\tilde x_n(t))$ and the Jacobi fields $Y_i(t+r_i^{k(l)})$ converges to a Jacobi filed $\tilde Y_i(t)$ such that
$\tilde Y_i(0)=0$ and $\tilde Y'_i(0)$ is a multiple of $\flat(\partial/\partial H_i)$. Let $t=T>0$ be the first zero of $\tilde Y_i(t)$. Then 
$\tilde x_i(T)=\lim_{l\to\infty} x_i(r_i^{k(l)+1})$ and
\begin{gather*}
\lim_{l\to\infty}\big|x_i(r_i^{k(l)})-
x_i(s_i^{k(l)})\big|=
\lim_{l\to\infty}\big|x_i(r_i^{k(l)+1})-
x_i(s_i^{k(l)+1})\big|\ ,\\
|f_i(\tilde x_i(0))-b_i|=|f_i(\tilde x_i(T))-b_i|
\end{gather*}

Therefore we have 
\begin{equation*}
\int_0^{T}\left|\frac{df_i(\tilde x_i(t))}{dt}\right|=
2(a_{i-1}^--a_i^+)\ ,
\end{equation*}
which contradicts Proposition \ref{prop:main}.
\end{proof}
% prop 3.2 'ðŠg'£'µ'Ä'¨'­'±'ƁB
%%
\section{Conjugate locus}
Let $p_0=(x_{1,0},\dots,x_{n,0})\in M$ be a general point, i.e., a point which 
is not contained in any hypersurfaces $N_i$ $(0\le i\le n)$. We shall 
determine the shape of the conjugate locus of $p_0$. 
In view of Proposition~\ref{prop:isometry} we may assume $0<x_{i,0}<\alpha_i/4$ for any $i$
without loss of generality. 
Although the first conjugate locus is our 
primary concern, we shall also 
consider the
$k$-th conjugate locus for $1\le k\le n-1$ as well (see Introduction). The reason
of doing so is that the first $n-1$ conjugate loci can be viewed
as a scattered image of the first conjugate locus of a point of
the sphere of constant curvature, which is a one point with 
multiplicity $n-1$, provided $M$ is sufficiently close to the 
standard sphere in some sense. 
%We may assume $0<x_{i,0}<\alpha_i/4$ without loss of generality (cf.\,\cite{IK2}, \S2).

We shall parametrize the unit cotangent space $U^*_{p_0}M$
by $(n-1)$-torus as follows: Putting $f_{i,0}=f_i(x_{i,0})$,
\begin{gather*}
\xi_i=\epsilon_i\sqrt{(-1)^{i-1}\prod_{k=1}^{n-1}(f_{i,0}-b_k(u_k))}\ ,\\
b_k(u_k)=f_{k+1,0}(\cos u_k)^2+
f_{k,0}(\sin u_k)^2\ ,
\end{gather*}
$u=(u_1,\dots,u_{n-1})\in (\R/2\pi\Z)^{n-1}$, where the sign 
$\epsilon_i$ is chosen to be equal to that of $\cos u_i\sin u_{i-1}$
if $2\le i\le n-1$, that of $\cos u_1$ if $i=1$, and that of
$\sin u_{n-1}$ if $i=n$.  
We denote by $[u]\in U_{p_0}^*M$ the corresponding covector. Observe that $(b_1,\dots, b_{n-1})$ is, in the
case of the ellipsoid, essentially 
the same as the elliptic coordinates
$(\mu_1,\dots,\mu_{n-1})$ on $U_pM$ described in Introduction. Accordingly, we define submanifolds
(with boundary) $C_k^\pm$ $(1\le k\le n-1)$ of $U_{p_0}^*M$ by
\begin{equation*}
C_k^-=\{[u]\,|\,  u_k=0,\pi\},\qquad C_k^+=\{[u]\,|\, u_k=\pm\frac{\pi}2\}\,.
\end{equation*}
Then $C_{k-1}^-\cup C_k^+$ is equal to the great sphere $\xi_k=0$,
they are diffeomorphic to
\begin{gather*}
C_k^-\simeq  S^{k-1}\times \bar D^{n-1-k},\quad C_k^+\simeq \bar D^{k-1}\times S^{n-1-k}\,,
\end{gather*}
and the boundaries $\partial C_k^\pm$ satisfy 
\begin{gather*}
\partial C_k^+=\partial C_{k-1}^-= C_k^+\cap C_{k-1}^-\simeq S^{k-2}\times S^{n-1-k}\quad (2\le k\le n-1)\,,\\
\partial C_{n-1}^-=\emptyset=\partial C_1^+\,.
\end{gather*}

We shall denote by 
\begin{equation*}
t\mapsto \gamma(t,u)=(x_1(t,u),\dots,x_n(t,u))
\end{equation*}
the geodesic such that $\gamma(0,u)=p_0$ and $\flat((\partial\gamma/\partial t)(0,u))=[u]\in U^*_{p_0}M$. Put
 \begin{equation*}
 Y_i(t,u)=\frac{\partial\gamma}{\partial u_i}(t,u)
 \qquad (1\le i\le n-1)
 \end{equation*}
 and let $t=r_i(u)$ be the first zero of the Jacobi field 
 $t\mapsto Y_i(t,u)$ for $t>0$. 
 Note that the Jacobi fields $Y_i$ are identical with constant multiple of the ones
 defined in the previous section. 
 When $[u]\in C_{i-1}^-\cap C_i^+$, $Y_{i-1}(t,u)$ 
and $Y_i(t,u)$ vanish identically. So, in this case,
we use the Jacobi fields $Z_{i-1}(t)$ and $Z_i(t)$
defined in \cite[\S5]{IK2}  (see also \S7.2 in this paper), and define $t=r_k(u)$ as the first zero of $Z_k(t)$ for $t>0$ $(k=i,i-1)$.
Actually, $Z_{i-1}(t)$ and $Z_i(t)$ are also 
defined for $[u]$ near the points in $C_{i-1}^-\cap C_i^+$ and they are linear combinations of $Y_{i-1}(t,u)$ and $Y_i(t,u)$ there. 
Thus the functions $r_i(t)$ are continuous at any 
$[u]\in U_{p_0}^*M$.
 
In view of Proposition~\ref{prop:main} and \cite[Proposition 5.5]{IK2} we obtain the following proposition.
\begin{prop}\label{prop:conj}
$r_i(u)\le r_{i-1}(u)$ for any $u\in(\R/2\pi\Z)^{n-1}$, and 
the equality holds if and only if 
$[u]\in C_{i-1}^-\cap C_i^+$, i.e., $b_i(u_i)=b_{i-1}(u_{i-1})$.
\end{prop}
\begin{proof}
Let $t_i=t_i(u)$ be the value defined by the formula~\eqref{def:ti} for the geodesic $\gamma(t,u)$.
We proved in Proposition~\ref{prop:main}
and remarked in \eqref{primeineq2} that
\begin{equation*}
r_i(u)\le t_i(u)\le r_{i-1}(u);
\end{equation*}
and $r_i(u)= t_i(u)$ if and only if $u_i= \pm \pi/2$ and
$r_{i-1}(u)= t_i(u)$ if and only if $u_{i-1}= 0,\pi$,
provided
$b_1(u_1),\dots,b_{n-1}(u_{n-1})$ and $a_0,\dots, a_n$ are all distinct. We shall prove that $r_i(u)\ne t_i(u)$ if $u_i\ne \pm \pi/2$ and
$r_{i-1}(u)\ne t_i(u)$ if $u_{i-1}\ne 0,\pi$, not necessarily assuming $b_1,\dots,b_{n-1}$ and
$a_1,\dots,a_n$ are distinct, which will indicate
the proposition.

Let us observe the formula~\eqref{eq:diff3} for
$j$ being replaced by $i$:
\begin{equation}\label{eq:diff4}
\begin{aligned}
&\frac1{2}\sum_{l=1}^{n}
\int_{t_l}^{r_i}\frac{(-1)^lG(f_l)\,(f_l-a_n)}
{(f_l-b_i)\sqrt{(-1)^{l-1}\prod_{m=1}^{n-1}(f_l-b_m)}}
\left|\frac{\partial x_l(t,u)}{\partial t}\right|
\ dt\\
&+\frac{\partial}{\partial b_i}
\sum_{l=1}^n\int_{a_l^+}^{a_{l-1}^-}
\frac{(-1)^lG(\lambda)\,(\lambda-a_n)
A(\lambda)\ d\lambda}
{\sqrt{-\prod_{m=1}^{n-1}(\lambda-b_m)
\cdot\prod_{m=0}^n(\lambda-a_m)}}=0\ ,
\end{aligned}
\end{equation}
where $G(\lambda)=\prod_{m\ne i,i-1}(\lambda-b_m)$ and $b_m=b_m(u_m)$. This formula is effective for $u$ such that
$b_1(u_1),\dots,b_{n-1}(u_{n-1})$ and $a_1,\dots,a_n$
are all distinct. We now take any $u\in(\R/2\pi\Z)^{n-1}$ such that 
$u_i\ne \pm\pi/2$ (i.e., $b_i(u_i)<f_{i,0}$)
and take a sequence $u^k=(u_1^k,\dots,u_{n-1}^k)\in (\R/2\pi\Z)^{n-1}$ such that $u^k\to
u$ as $k\to \infty$ and such that $b_1(u_1^k),\dots,b_{n-1}(u_{n-1}^k)$ and $a_1,\dots,a_n$
are all distinct for any $k$. We also assume that the ordering of $a_j$ and $b_j(u_j^k)$ does not 
change when $k$ varies for each $j$.
By Proposition~\ref{prop:limitineq} we see that the second line of the formula~\eqref{eq:diff4}
for $(b_m)=(b_m(u_m^k))$ has a limit value as $k\to\infty$ and the value is still positive.
Also, each summand in the first line of the formula~\eqref{eq:diff4} is positive if $l\ne i$ and
is negative if $l=i$ for any $k$ by Proposition~\ref{prop:newineq}. Therefore, there is a constant $c>0$
such that
\begin{equation}\label{ineq:ri}
\int_{r_i(u^k)}^{t_i(u^k)}\frac{(-1)^iG(f_i)\,(f_i-a_n)}
{(f_i-b_i^k)\sqrt{(-1)^{i-1}\prod_{m=1}^{n-1}(f_i-b_m^k)}}
\left|\frac{\partial x_i(t,u^k)}{\partial t}\right|\,dt\ge c
\end{equation}
for sufficiently large $k$, where $b_m^k=b_m(u_m^k)$.
We note that 
\begin{equation*}
f_i(x_i(t_i(u^k),u^k))=f_{i,0},\quad
b_i^k<f_{i,0},\quad a_i<f_{i,0}<a_{i-1}\,.
\end{equation*}

We now consider the following two cases separately: $f_{i,0}<b_{i-1}(u_{i-1})$; 
and $f_{i,0}=b_{i-1}(u_{i-1})$. Let us first assume that $f_{i,0}<b_{i-1}(u_{i-1})$.
By the formula~\eqref{xprime} we see that there are (sufficiently small) constant $\delta>0$ and
positive constants $c_1, c_2$ not depending on $k$ such that 
\begin{equation}\label{boundxi}
c_1\le \left|\frac{\partial x_i(t,u^k)}{\partial t}\right|\le c_2\quad\text{if}\quad
|f_i(x_i(t,u^k))-f_{i,0})|\le \delta
\end{equation}
for any (sufficiently large) $k$. Since $|df_i/dx_i|$ is bounded both above and below if $(a_i-f_i)(f_i-a_{i-1})$ is bounded away from $0$ in view of
\eqref{eq:fdiff}, we also have
\begin{equation}
c'_1\le \left|\frac{\partial f_i(x_i(t,u^k))}{\partial t}\right|\le c'_2\quad\text{if}\quad
|f_i(x_i(t,u^k))-f_{i,0})|\le \delta\,,
\end{equation}
for some positive constants $c_1'$ and $c_2'$ not
depending on $k$. Thus the map $t\mapsto f_i(x_i(t,u^k))$ and its inverse are ``equicontinuous''
in $k$ on a neighborhood of $t=t_i(u)$.
Now, suppose that $r_i(u)=t_i(u)$. Then for large $k$ we have
\begin{equation*}
|f_i(x_i(t,u^k))-f_{i,0})|\le \delta\quad\text{for any}\quad t\in[r_i(u^k),t_i(u^k)]
\end{equation*}
and by \eqref{ineq:ri},
\begin{equation}\label{ineq:contradict}
c'(t_i(u^k)-r_i(u^k))\ge c
\end{equation}
for some constant $c'>0$ not depending on $k$. Thus, taking a limit $k\to\infty$, we have a contradiction.

Next, suppose that $f_{i,0}=b_{i-1}(u_{i-1})$. In this case $b_{i-1}(u_{i-1})$ is not equal to $b_{i-2}(u_{i-2})$, because
\begin{equation*}
b_{i-1}(u_{i-1})=f_{i,0}<a_{i-1}\le b_{i-2}(u_{i-2})\,.
\end{equation*}
Therefore we may assume that $u^k$ are chosen so 
that $u_{i-1}^k=u_{i-1}$ for any $k$.
Then $t=t_i(u^k)$ is the turning point of the functions $x_i(t,u^k)$ and $f_i(x_i(t,u^k))$ (including the case $k=\infty$), and by the same reason as in the previous case, the functions
\begin{equation*}
\frac{|\partial x_i(t,u^k)/\partial t|}{\sqrt{b_{i-1}(u_{i-1})-f_i(x_i(t,u^k))}},\qquad
\frac{|\partial f_i(x_i(t,u^k))/\partial t|}{\sqrt{b_{i-1}(u_{i-1})-f_i(x_i(t,u^k))}}\end{equation*}
are bounded both above and below by positive constants not depending on $k$ if 
\begin{equation}\label{bounddelta}
|f_i(x_i(t,u^k))-f_{i,0})|\le \delta,
\end{equation}
where $\delta$ is a constant not depending on $k$.
Thus, as in the previous case, we see that there
is a constant $\delta'>0$ not depending on $k$
such that the inequality \eqref{bounddelta} holds
for $t$ with 
\begin{equation}
|t-t_i(u^k)|\le  \delta'
\end{equation}
for any (sufficiently large) $k$.
Therefore, if we assume $r_i(u)=t_i(u)$, then
we again have the inequality \eqref{ineq:contradict} for large $k$, which is a contradiction.

We have thus proved that $r_i(u)<t_i(u)$ if $u_i\ne \pm\pi/2$. The implication of $t_i(u)<r_{i-1}(u)$
for $u$ with $u_{i-1}\ne 0,\pi$ is similar to the above, and
we omit the detail.
\end{proof}
\begin{thm}\label{thm:conj}
\begin{enumerate}
\item $t=r_{n-1}(u)$ represents the first conjugate point of $p_0$ 
 along the geodesic $\gamma(t,u)$. 
\item If the Riemannian manifold $M$ is close to 
the round sphere so that the second zero $t=r^2_{n-1}(u)$ of the Jacobi field $Y_{n-1}(t,u)$ is greater than
$r_1(u)$, then $t=r_i(u)$ represents the $(n-i)$-th conjugate point of $p_0$ 
 along the geodesic $\gamma(t,u)$ for $2\le i\le n-1$. 
\end{enumerate}
\end{thm}
\begin{proof}
The assertions are immediate from Proposition \ref{prop:conj}. For (2), we remark 
that for the round sphere the first zeros of
$Y_i(t,u)$ $(1\le i\le n-1)$ coincide, and the
second zero of $Y_{n-1}(t,u)$ is greater than them. 
Therefore, if $M$ is ``close to the round sphere'',
then the second zero of $Y_{n-1}(t,u)$ appears
after the first zero of $Y_{1}(t,u)$.
\end{proof}
We put
\begin{align*}
\tilde{K}_{i}(p_0)&=\{r_i(u)\sharp [u]\in 
T_{p_0}M\ |\ u\in(\R/2\pi\Z)^{n-1}\}\\
K_{i}(p_0)&=\{\gamma(r_i(u),u)\ |\ u\in(\R/2\pi\Z)^{n-1}\}\ .
\end{align*}
Then $K_{n-1}(p_0)$ ($\tilde K_{n-1}(p_0)$)
represents the first (tangential) conjugate locus
of $p_0$, and if $M$ is close to the round sphere, then $K_{n-j}(p_0)$ ($\tilde K_{n-j}(p_0)$)
represents the $j$-th (tangential) conjugate locus
of $p_0$ for $2\le j\le n-1$.
For the smoothness of the functions $r_i(u)$ we have the following
\begin{lemma}
$u\mapsto r_i(u)$ is of $C^\infty$ for $[u]\not\in \partial C_i^{\pm}$, i.e., for $u$ with 
$b_i(u_i)\ne b_{i-1}(u_{i-1})$ and $b_i(u_i)\ne b_{i+1}(u_{i+1})$.
\end{lemma}
\begin{proof}
Under the given assumption, $Y_i(t,u)$ is written as $y_i(t,u)V_i(t,u)$, as described in
the previous section. Since $(\partial/\partial t)y_i(t,u)\ne 0$ at $t=r_i(u)$, the lemma follows from the implicit function theorem.
\end{proof}

\begin{prop}\label{prop:reg}
\begin{enumerate}
\item 
$\displaystyle\frac{\partial}{\partial u_i}r_i(u)\ne 0$ if 
$\displaystyle [u]\not\in C_i^\pm$.
\item $\displaystyle\frac{\partial}{\partial u_i}r_i(u)= 0$ 
and $\displaystyle\frac{\partial^2}{\partial u_i^2}r_i(u)\ne 0$
for $[u]\in C_i^\pm$, $[u]\not\in \partial C_i^\pm$.
\end{enumerate}
\end{prop}
\begin{proof}
(1) \,First we assume that $b_1(u),\dots,b_{n-1}(u)$ and $a_0,\dots,a_n$ are all distinct. 
Differentiating the formula \eqref{eq:diff3} ($j$ being $i$ here) in the proof of Proposition~\ref{prop:main}
in terms of the deformation parameter defining $cY_{j}$ once again, we have
\begin{equation}\label{derri}
\begin{aligned}
&\frac34\sum_{l=1}^{n}
\int_{t_l}^{r_i}\frac{(-1)^lG(f_l)}
{(f_l-b_i)^2\sqrt{(-1)^{l-1}\prod_{k=1}^{n-1}(f_l-b_k)}}
\left|\frac{\partial x_l(t,u)}{\partial t}\right|
\ dt\\
&+2\frac{\partial^2}{\partial b_i^2}
\sum_{l=1}^n\int_{a_l^+}^{a_{l-1}^-}
\frac{(-1)^lG(\lambda)
A(\lambda)\ d\lambda}
{\sqrt{-\prod_{k=1}^{n-1}(\lambda-b_k)
\cdot\prod_{k=0}^n(\lambda-a_k)}}\\
&+\frac{c}2\frac{\partial r_i}{\partial u_i}\sum_{l=1}^n \frac{(-1)^lG(f_l)}
{(f_l-b_i)\sqrt{(-1)^{l-1}\prod_{k=1}^{n-1}(f_l-b_k)}}
\left|\frac{\partial x_l(t,u)}{\partial t}\right|_{t=r_i(u)}=0\ ,
\end{aligned}
\end{equation}
where $G(\lambda)=(\lambda-a_n)\prod_{k\ne i}(\lambda-b_k)$,
\begin{equation*}
c=\left(\frac{db_i}{du_i}\right)^{-1}=\frac1{2\sin u_i\cos u_i (f_{i,0}-f_{i+1,0})}\,,
\end{equation*}
 and $f_l$ in the third line is equal to $f_l(x_l(r_i(u),u))$. Since $r_i<t_l$, $f_l-b_i>0$ for
$l\le i$, and $r_i>t_l$, $f_l<b_i$ for $l\ge i+1$, the first line of the above formula is
positive; while the second line is also positive by Proposition~\ref{prop:newineq} (3).
Therefore it follows that $\partial r_i/\partial u_i\ne 0$.

Next we consider the general case. As before, we take a sequence $u^k\in(\R/2\pi\Z)^{n-1}$
such that $u^k\to u$ as $k\to\infty$ and that
$b_j(u_j^k)$ $(1\le j\le n-1)$ and $a_l$ $(0\le l\le n)$ are all distinct for each $k$.
Let us consider the formula \eqref{derri} for 
$u^k$ and take a limit $k\to \infty$.
The second line then converges to a positive value 
by Proposition~\ref{prop:limitineq} and the first
line is positive for each $k$. For the third line,
we note that 
\begin{equation*}
f_l(x_l(r_i(u),u))\ne b_i(u_i)\quad (l=i,i+1)
\end{equation*}
by the proof of Proposition~\ref{prop:conj} and
\begin{equation*}
\sqrt{\prod_{m=1}^{n-1}|f_l(x_l(r_i(u^k),u^k))-b_m(u_m^k)|}
\left|\frac{\partial x_l}{\partial t}(r_i(u^k),u^k)\right|\le 1\,,
\end{equation*}
since
\begin{equation*}
\sqrt{\prod_{\substack{1\le m\le n\\m\ne l}}|f_m-f_l|}\, 
\left|\frac{\partial x_l}{\partial t}(r_i(u^k),u^k)\right|\le 1
\end{equation*}
by the expression of the metric \eqref{eq:metric} ($f_m=f_m(x_m(r_i(u^k),u^k))$) and
\begin{equation*}
|f_l-b_m|\le \begin{cases}
|f_l-f_m|\quad &(1\le m\le l-1)\\
|f_l-f_{m+1}|\quad &(l\le m\le n-1)
\end{cases}\,. 
\end{equation*}
Therefore the integral in the third line of \eqref{derri} for $u^k$ remains finite as $k\to\infty$.
Those facts indicate that $(\partial r_i/\partial u_i)(u)$ does not vanish.

(2) Let us consider the case where $u_i=0$ ($u_{i+1}\ne \pm\pi/2$). We describe \eqref{eq:diff3}
for $u$ with $u_i<0$ and $u_i$ being close to $0$ in the form:
\begin{equation}\label{eq:cusp}
\begin{gathered}
\frac12\sum_{\substack{1\le l\le n\\l\ne i+1}}\int_{t_l}^{r_i}\frac{(-1)^lG(f_l)}
{(f_l-b_i)\sqrt{(-)^{l-1}\prod_{k=1}^{n-1}(f_l-b_k)}}
\left|\frac{dx_l(t)}{dt}\right|
\ dt\\
+\frac14\int_{f_{i+1}(x_{i+1}(r_i(u),u))}^{f_{i+1}(x_{i+1,0})}\frac{(-1)^{i+1}G(\lambda)A(\lambda)\,d\lambda}
{(\lambda-b_i)\sqrt{-\prod_{k=1}^{n-1}(\lambda-b_k)
\cdot\prod_{k=0}^n(\lambda-a_k)}}\\
+\frac{\partial}{\partial b_i}\sum_{l=1}^n\int_{a_l^+}^{a_{l-1}^-}
\frac{(-1)^lG(\lambda)A(\lambda)\ d\lambda}
{\sqrt{-\prod_{k=1}^{n-1}(\lambda-b_k)
\cdot\prod_{k=0}^n(\lambda-a_k)}}=0\,,
\end{gathered}
\end{equation}
where $G=(\lambda-a_n)\prod_{k\ne i,i+1}(\lambda-b_k)$.
Note that this formula is effective for general 
$u$, i.e., for $u$ not
necessarily satisfying that $b_l(u_l)$ $(1\le l\le 
n-1)$ and $a_m$ $(0\le m\le n)$ are all distinct.
In fact, since the second and the third line have
definite values at such $u$, so is the first line.
Note also that in this case $b_i=a_i^-$. Since we are assuming $0<x_{i+1,0}<\alpha_{i+1}/4$, we have
\begin{equation*}
f'_{i+1}(x_{i+1}(t_{i+1}(u),u))=f'_{i+1}(x_{i+1,0})>0\,,
\end{equation*}
and since $u_i<0$ and $u_i$ is close to $0$, we also have
\begin{equation*}
t_{i+1}(u)<r_i(u),\quad \frac{\partial x_{i+1}}{\partial t}(t,u)<0\,.
\end{equation*}
Therefore, 
\begin{equation*}
f_{i+1}(x_{i+1}(r_i(u),u)) < f_{i+1}(x_{i+1,0})
<b_i(u_i)
\end{equation*}
when $u_i< 0$, and they all coincide when $u_i=0$, by \eqref{primeineq}, \eqref{primeineq2}, and Proposition~\ref{prop:main}.

We denote by $\Phi(\lambda)$ the integrand of the
second line in the formula \eqref{eq:cusp}:
\begin{equation*}
\Phi(\lambda)=\frac{(-1)^{i+1}G(\lambda)A(\lambda)}
{(\lambda-b_i)\sqrt{-\prod_{k=1}^{n-1}(\lambda-b_k)
\cdot\prod_{k=0}^n(\lambda-a_k)}}\,.
\end{equation*}
When $\lambda$ is in the interval of the integration;
\begin{equation*}
f_{i+1}(x_{i+1}(r_i(u),u)) \le \lambda\le f_{i+1}(x_{i+1,0})\,,
\end{equation*}
then $\Phi(\lambda)<0$ and 
\begin{equation*}
-\Phi(\lambda)\le c\,\left(b_i(u_i)-f_{i+1}(x_{i+1,0})\right)^{-\frac32}=\frac{c'}{|\sin u_i|^3}
\end{equation*}
for some positive constants $c$, $c'$. 
Thus one obtains
\begin{equation}\label{eq:cusp2}
0<-\int_{f_{i+1}(x_{i+1}(r_i(u),u))}^{f_{i+1}(x_{i+1,0})} \Phi(\lambda)\,d\lambda\le \frac{c\left|x_{i+1,0}-x_{i+1}(r_i(u),u)\right|}{|\sin u_i|^3}\,.
\end{equation}
for some (another) constant $c>0$.

Now we need the following lemma, which is essentially the same as \cite[Lemma~8.2]{IK3}.
\begin{lemma}\label{cusp3}
Regarded as a function of $u_i$ (other $u_j$'s being fixed),
\begin{gather*}
x_{i+1}(r_i(u),u)-x_{i+1,0}=c\,u_i^3+O(u_i^4)\,,\\
c=\frac13
\left(\frac{\partial^2 x_{i+1}}{\partial t\partial u_i}(r_i(u),u)
\frac{\partial^2 r_i}{\partial u_i^2}(u)\right)\bigg|_{u_i=0}\,.
\end{gather*}
\end{lemma}
\begin{proof}
We have 
\begin{equation*}
\frac{\partial}{\partial u_i}x_{i+1}(r_i(u),u)=
\frac{\partial x_{i+1}}{\partial t}(r_i(u),u)
\,\frac{\partial r_i}{\partial u_i}(u)\,.
\end{equation*}
Since $(\partial r_i/\partial u_i)(u)=(\partial x_{i+1}/\partial t)(r_i(u),u)=0$ when $u_i=0$, it therefore follows that
\begin{equation}\label{thirdder}
\frac{\partial^3}{\partial u_i^3}x_{i+1}(r_i(u),u)\big\vert_{u_i=0}=
2\,\frac{\partial^2 x_{i+1}}{\partial t\partial u_i}
(r_i(u),u)\,\frac{\partial^2 r_i}{\partial u_i^2}(u)\big|_{u_i=0}\,,
\end{equation}
which indicates the lemma.
\end{proof}
We continue the proof of Proposition~\ref{prop:reg}. Assume that
\begin{equation*}
\frac{\partial^2 r_i}{\partial u_i^2}(u)\big|_{u_i=0}=0\,.
\end{equation*}
Then by \eqref{eq:cusp2} and Lemma~\ref{cusp3} the second line of the formula \eqref{eq:cusp}
tends to $0$ when $u_i\to 0$. However, the first line of the formula \eqref{eq:cusp} is nonnegative and the third line us positive by Proposition~\ref{prop:newineq} (2) and Proposition~\ref{prop:limitineq}, which is a contradiction. Thus we have
\begin{equation*}
\frac{\partial^2 r_i}{\partial u_i^2}(u)\big|_{u_i=0}\ne 0
\end{equation*}
in this case. The case where $u_i=\pi$ is similar.

For the cases where $u_i=\pm\pi/2$, $b_i=a_i^+$ and one should consider the integral
concerning the variable $x_i$ in the formula \eqref{eq:cusp} instead of that concerning the variable $x_{i+1}$ as above. Then the argument is parallel as above and we shall omit the detail. This finishes th proof of Proposition~\ref{prop:reg}.
\end{proof}

We remark that in the above proof we have actually proved the following fact.
\begin{cor}\label{cusp4}
The constant $c$ which appeared in Lemma~\ref{cusp3} does not vanish.
\end{cor}

Thus, as a consequence of Proposition~\ref{prop:reg} and Corollary~\ref{cusp4}, we have the following theorem.
\begin{thm}\label{thm:cusp}
The following statements hold for each $i$ $(1\le i\le n-1)$. For $i\ne n-1$, we assume 
that the second zero $t=r_{n-1}^2(u)$ of the Jacobi field $Y_{n-1}(t,u)$ is greater than
$r_1(u)$ for any $u\in(\R/2\pi\Z)^{n-1}$.
\begin{enumerate}
\item The map $u\mapsto 
\gamma(r_i(u),u)$  is an immersion at $[u]$ with $[u]\not\in C_i^\pm$. 
In particular, $K_{i}(p_0)$ is a smooth hypersurface around
such points $\gamma(r_i(u),u)$.
\item For each $p=\gamma(r_i(u),u)$ with $[u]\in C_i^\pm$, $[u]\not\in \partial C_i^\pm$, 
there is a neighborhood $U$ of $p$ and a function $x,y$ on $U$
such that $dx\wedge dy\ne 0$ and $U\cap K_{i}(p_0)$ is given by $x^3=y^2$ and such that the edge of
vertices $x=y=0$ corresponds to $\{\gamma(r_i(u),u)\,|\,[u]\in C_i^\pm\}$. Namely, $K_{i}(p_0)$  is diffeomorphic to a cuspidal edge around $p$.
\end{enumerate}
\end{thm}
\begin{proof}
(1) By the assumption, we see that $n-1$ vectors
\begin{equation*}
\frac{\partial}{\partial u_i}\gamma(r_i(u),u)=\frac{\partial \gamma}{\partial t}(r_i(u),u)\frac{\partial r_i}{\partial u_i}(u)
\end{equation*}
and
\begin{equation*}
\frac{\partial}{\partial u_k}\gamma(r_i(u),u)=\frac{\partial \gamma}{\partial t}(r_i(u),u)\frac{\partial r_i}{\partial u_k}(u)+\frac{\partial \gamma}{\partial u_k}(r_i(u),u)\quad(k\ne i)
\end{equation*}
are linearly independent. Therefore the map $u\to \gamma(r_i(u),u)$ is an immersion.

(2) We fix $u_0=(u_{1,0},\dots,u_{n-1,0})$ such that
$[u_0]\in C_i^\pm$, $[u_0]\not\in \partial C_i^\pm$. We consider the case where $u_{i,0}=0$. Other cases ($u_{i,0}=\pi,\pm\pi/2$) will be similar.
From the assumption the $n-1$ vectors
\begin{equation*}
\frac{\partial\gamma}{\partial t}(r_i(u),u),\quad\frac{\partial\gamma}{\partial u_k}(r_i(u),u)\quad(k\ne i)
\end{equation*}
are linearly independent at $u=u_0$ and, by Proposition~\ref{prop:dualtob}, their $dx_{i+1}$-components vanish. Therefore, we can take a coordinate system $(z_1,\dots,z_n)$
around the point $p=\gamma(r_i(u_0),u_0)$ such that
\begin{equation*}
z_1=x_{i+1},\quad dz_2\left(\frac{\partial \gamma}{\partial t}(r_i(u_0),u_0)\right)\ne 0,\quad z_k(p)=0\text{ for any }k,
\end{equation*}
and the Jacobian of the map
\begin{equation*}
u\mapsto (z_3,\dots,z_n,u_i)
\end{equation*}
given by $z_k=z_k(\gamma(r_i(u),u))$ does not vanish
at $u=u_0$. Then, putting 
\begin{equation*}
u'=(u_1,\dots,u_{i-1},u_{i+1},\dots,u_{n-1}),
\end{equation*}
we have
\begin{equation*}
z_1(\gamma(r_i(u),u))=c_1(u_i,u')u_i^3,\quad
z_2(\gamma(r_i(u),u))=c_2(u_i,u')u_i^2\,,
\end{equation*}
where $c_1$ and $c_2$ are functions of $u=(u_i,u')$ which do not vanish at $u_0=(0,u'_0)$. We may
assume that $c_2$ is positive at $u_0$. Thus,
we can replace the coordinate function $u_i$
with
\begin{equation*}
v_i=\sqrt{c_2(u_i,u')}u_i
\end{equation*}
so that $(v_i,u')$ is the new coordinate system on
$U_{p_0}^*M$, and we have, putting $z_k(v_i,u')=
z_k(\gamma(r_i(u),u))$,
\begin{equation*}
z_1(v_i,u')=c_3(v_i,u')v_i^3,\quad z_2(v_i,u')=v_i^2
\end{equation*}
and $c_3(0,u'_0)\ne 0$. 

Since the map
\begin{equation*}
(v_i,u')\mapsto (z_3(v_i,u'),\dots,z_n(v_i,u'),v_i)
\end{equation*}
is a local diffeomorphism around the point
$(v_i,u')=(0,u'_0)$, we can take the inverse function so that $(v_i,u')$ is a function of
$(z_3,\dots,z_n,v_i)$. Therefore, the map
$u\mapsto\gamma(r_i(u),u)$ is described as a map
\begin{equation*}
(z_3,\dots,z_n,v_i)\mapsto(z_1,\dots,z_n)
\end{equation*}
such that
\begin{equation*}
z_1=c_4(z_3,\dots,z_n,v_i)v_i^3,\quad z_2=v_i^2\,.
\end{equation*}
and $c_4(0,\dots,0)\ne 0$.

We put $z'=(z_3,\dots,z_n)$ and
\begin{equation*}
c_{4,\pm}(z',v_i)=\frac12(c_4(z',v_i)\pm c_4(z',-v_i))\,.
\end{equation*}
Then, since $c_{4,+}$ is a even function in $v_i$,
there is a $C^\infty$ function $c_5$ such that
\begin{equation*}
c_{4,+}(z',v_i)=c_5(z',v_i^2).
\end{equation*}
Similarly, we have
\begin{equation*}
c_{4,-}(z',v_i)=v_i\,c_6(z',v_i^2)
\end{equation*}
for some $C^\infty$ function $c_6$.
Thus we have
\begin{gather*}
z_1=(c_5(z',v_i^2)+v_i\,c_6(z',v_i^2))\,v_i^3\\
=c_5(z',z_2)\,v_i^3+c_6(z',z_2)\,z_2^2
\end{gather*}
Therefore, replacing the coordinate function $z_1$
with
\begin{equation*}
\bar z_1=\frac{z_1-c_6(z',z_2)\,z_2^2}{c_5(z',z_2)}\,,
\end{equation*}
we see that the map $u\mapsto \gamma(r_i(u),u)$
is expressed as the map
\begin{equation*}
(z',v_i)\mapsto (\bar z_1,z_2,z')
\end{equation*}
such that
\begin{equation*}
\bar z_1=v_i^3,\quad z_2=v_i^2\,.
\end{equation*}
Thus the theorem has been proved.
\end{proof}

%%%%%
\section{Singularities arising at points with double conjugacy}
\subsection{Definition of $D_4^+$ Lagrangian singularity}
We first review the notion of Lagrangian singularity and that of generating family which
describes a Lagrangian submanifold. After that, we state the definition of $D_4^+$ Lagrangian
singularity. For the statements of this subsection we refer to \cite{AGV} for Lagrangian singularities
and \cite{W} for versal deformations.

\subsubsection*{Lagrangian singularity}
Let $N$ be a manifold and let $L$ be a Lagrangian submanifold of $T^*N$.
A {\it Lagrangian singularity} is a singularity of the map $\pi\circ i: L\to N$, where
$i$ and $\pi$ denote the inclusion $L\to T^*N$ and the bundle projection $T^*N\to N$
respectively. More precisely, for points $\lambda_0\in L$ and $q_0=\pi(\lambda_0)\in N$,
we consider the ``Lagrangian equivalence class'' of the map-germ $(\pi\circ i): (L,\lambda_0) \to (N,q_0)$. Two such map-germs
$(L,\lambda_0) \to (N,q_0)$ and $(\pi'\circ i'): (L',\lambda'_0) \to (N',q'_0)$ are said to be 
Lagrangian equivalent if there is a diffeomorphism $\phi:(N,q_0)\to (N',q'_0)$ and a symplectic 
diffeomorphism
$\Phi:(T^*N,\lambda_0)\to (T^*N', \lambda'_0)$ such that the diagram
\begin{equation*}
\begin{CD}
(T^*N,\lambda_0) @>\Phi>> (T^*N',\lambda'_0)\\
@V\pi VV	@VV\pi' V\\
(N,q_0) @>>\phi > (N',q'_0)
\end{CD}
\end{equation*}
is commutative and such that $\Phi(L,\lambda_0)=(L',\lambda'_0)$. Actually, $\Phi$ is described as 
\begin{equation*}
\Phi(\lambda)=(\phi^*)^{-1}(\lambda)+dh_{\phi(\pi(\lambda))}\,,\quad \lambda\in T^*N
\end{equation*}
for some function $h$ on $N'$ in this case.

\subsubsection*{Generating family}
Let $(L,\lambda_0)\subset T^*N$ and $(N,q_0)$ be as above. Let $x=(x_1,\dots,x_n)$ be a coordinate system on $N$ so that $q_0$ corresponds to $a=(a_1,\dots,a_n)$ $(n=\dim N)$. A function $F(u,x)=F(u_1,\dots,u_k,x_1,\dots,x_n)$ defined on a neighborhood of 
$(b,a)\in \R^k\times \R^n$ is called a ``generating family'' for $L$ around
the reference point $\lambda_0\in L$ if it satisfies
\begin{enumerate}
\item $0\in \R^k$ is a regular value of the map 
\begin{equation*}
d_uF: (u,x)\mapsto (\partial F/\partial u_1,\dots,\partial F/\partial u_k)
\end{equation*}
and $d_uF(b,a)=0$. Thus $C=(d_uF)^{-1}(0)$ is a $n$-dimensional manifold and $(b,a)\in C$.
\item The map
\begin{equation*}
d_xF: C\ni (u,x)\mapsto \sum_{l=1}^n(\partial F/\partial x_l)(u,x)\,dx_l\in T^*_xN\subset T^*N
\end{equation*}
gives an embedding of $C$ into $T^*N$ whose image is $L$ (a neighborhood of $\lambda_0$) and $d_xF(b,a)=\lambda_0$.
\end{enumerate}
It can be seen that 
\begin{equation*}
k\ge \dim\ker \left((\pi\circ i)_*: T_{\lambda_0}L\to T_{q_0}N\right)\,.
\end{equation*}
If the equality holds, then the generating family is called {\it minimal}.
A way of obtaining a minimal generating family is as follows: Let $(x, \xi)$ be the canonical coordinate system of $T^*N$ associated with a coordinate system $x=(x_1,\dots, x_n)$ on $N$.
One can choose $x$ so that 
\begin{equation*}
(\pi\circ i)^*(dx_{j})= 0\quad \text{at } \lambda_0\quad(1\le j\le k),  
\end{equation*}
where $k=\dim\ker ((\pi\circ i)_*)_{\lambda_0}$. Then $(\xi_1,\dots,\xi_k,x_{k+1},\dots,x_n)$
form a coordinate system of $L$ around $\lambda_0$ and
\begin{equation*}
-\sum_{i=1}^kx_id\xi_i+\sum_{j=k+1}^n\xi_jdx_j
=\sum_{i=1}^n\xi_idx_i-d\left(\sum_{i=1}^k \xi_ix_i\right)
\end{equation*}
is a closed form on $L$. Thus there is a function 
\begin{equation*}
\hat F=\hat F(\xi_1,\dots,\xi_k,x_{k+1},\dots,x_n)
\end{equation*}
on $L$ so that 
\begin{equation*}
\partial \hat F/\partial \xi_i=-x_i|_L\,, \quad \partial \hat F/\partial x_j=\xi_j|_L\qquad (1\le i\le k,\ k+1\le j\le n)\,.
\end{equation*}
Then
\begin{equation*}
F(\xi_1,\dots,\xi_k,x_1,\dots,x_n)=\sum_{i=1}^k\xi_ix_i +\hat F(\xi_1,\dots,\xi_k,x_{k+1},\dots,x_n)
\end{equation*}
is the desired minimal generating family.

Let $G(v_1,\dots,v_{k'},y_1,\dots,y_n)$ with the base point $(b',a')$ be another minimal generating family for a Lagrangian submanifold 
$(\tilde L, \tilde \lambda_0)\subset T^*\tilde N$. Then
those two minimal generating families are said to be $\mathcal R^+$-equivalent
if $k'=k$ and there is a diffeomorphism $\Psi:\R^k\times\R^n\to \R^k\times\R^n$ $((b,a)\mapsto
(b',a'))$ of the form
\begin{equation*}
\Psi(u,x)=(\psi(u,x),\phi(x))
\end{equation*}
and a function $h(x)$ so that $F(u,x)=G(\Psi(u,x))+h(x)$. The following criterion is crucial
(see the theorem in \cite[p.304]{AGV} and its proof).
\begin{thm}
Two minimal generating families $F(u,x)$ and $G(v,y)$ are $\mathcal R^+$-equivalent
if and only if the corresponding Lagrangian submanifolds $(L,\lambda_0)\subset T^*N$ and 
$(\tilde L,\tilde \lambda_0)\subset T^*\tilde N$ are Lagrangian equivalent.
\end{thm}

\subsubsection*{Versal deformation of a function germ}
Let $F(u,x)$ be a function germ on $\R^k\times\R^n$ at $(b,a)$ and put
\begin{equation*}
f(u)=f(u_1,\dots, u_k)=F(u,a).
\end{equation*}
Such $F$ is called a deformation (or an unfolding) of the function germ $(f(u), b)$. We are interested in the case where $F(u,x)$ is a {\it versal} deformation of $f$. We do not explain the original definition of versality here; the following characterization by Mather is enough for our purpose (see \cite[\S3]{W} for
the proof of the next two theorems and a detailed explanation on the theory of versal deformations).
\begin{thm}
The function germ $(F(u,x),(b,a))$ is a versal deformation of the function germ $(f(u), b)$
if and only if the quotient space
\begin{equation*}
\mathcal{E}_k\big/\left(\frac{\partial f}{\partial u_1},\dots,\frac{\partial f}{\partial u_k}\right)
\end{equation*}
is spanned by elements represented by constant functions and
\begin{equation*}
\frac{\partial F}{\partial x_j}(u,a)\quad (1\le j\le n)
\end{equation*}
as a vector space. 
\end{thm}
Here $\mathcal E_k$ denotes the algebra of function germs in $(u_1,\dots,u_k)$ at $u=b$ and 
$(\dots,\partial f/\partial x_j,\dots)$ stands for its ideal generated by 
$\partial f/\partial x_j$ $(1\le j\le k)$.
\begin{thm}
Let $(F(u,x),(b,a))$
and $(H(v,y),(b',a'))$
be two deformation germs on $\R^k\times\R^n$
of $f(u)=F(u,a)$ and $h(v)=H(v,a')$ respectively. 
Suppose $F$ and $H$ are versal deformations. 
Then the two deformation germs $F$ and $H$ are $\mathcal R^+$-equivalent if and only if the
function germs $(f(u),b)$ and $(h(v),b')$ are 
equivalent, i.e., there is a diffeomorphism germ
$\phi:(\R^k,b)\to(\R^k,b')$ and a constant $c\in\R$ such that $f=h\circ \phi+c$.
\end{thm}
The $\mathcal R^+$-equivalence in the above theorem is the same as that
for generating families. If $(F(u,x),(b,a))$ is 
a versal deformation of $(f(u),b)$, then it is known that the function germ $f(u)$ is {\it finitely determined}, i.e., there is a positive integer $l$ such that any function germ $(h(u),b)$
whose $l$-jet is equal to the $l$-jet of $f(u)$ at $b$ is equivalent to $(f(u), b)$. 
(In this case $(f(u),b)$ is said to be $l$-determined. ) 
Therefore we 
have the following criterion for Lagrangian 
equivalence of Lagrangian singularities.
\begin{thm}\label{criterion}
Let $(F(u,x), (b,a))$, a function germ on $\R^k\times \R^n$, be a minimal generating family for
a Lagrangian submanifold $(L,\lambda_0)\subset T^*N$. Suppose $F$ is a versal deformation of $f(u)=F(u,a)$ at $b$ and $f(u)$ is $l$-determined.
Let $H(v,y), (b',a'))$ be another function germ on
$\R^k\times\R^n$ and is a minimal generating 
family of a Lagrangian submanifold $(L',\lambda'_0)\subset T^*N'$. Suppose also that $H$ is a versal
deformation o $h(v)=H(v,a')$ at $b'$. 
Then the Lagrangian
singularity $\pi\circ i:(L,\lambda_0)\to (N,q_0)$
is Lagrangian equivalent to $\pi'\circ i:(L',\lambda'_0)\to (N',q'_0)$ if and only if there
is a diffeomorphism germ $\phi:(\R^k,b)\to(\R^k,b')$ 
and a constant $c\in\R$ such that
the $l$-jets of $h(\phi(u))+c$ and $f(u)$ at $b$
coincide.
\end{thm}

\subsubsection*{$D_4^+$ singularity}
The equivalence class of the function germ $f(u_1,u_2)=u_1^3+u_1u_2^2$ at
$0\in\R^2$ is called the $D_4^+$ singularity.
It is $3$-determined and the quotient space 
\begin{equation*}
\mathcal{E}_2\big/\left(\frac{\partial f}{\partial u_1},\frac{\partial f}{\partial u_2}\right)
\end{equation*}
is spanned by $1,u_1,u_2$, and $u_2^2$. Put
\begin{equation*}
F(u_1,u_2,x_1,\dots,x_n)=u_1^3+u_1u_2^2+x_1u_1+x_2u_2
+x_3u_2^2+\sum_{j=4}^nc_jx_j\,,
\end{equation*}
where $c_4,\dots,c_n\in\R$.
Then $(F(u,x),(0,0))$ is a versal deformation of
$(f(u),0)$. Putting 
\begin{equation*}
C=\{(u,x)\,|\,\partial F/\partial u_1=\partial F/\partial u_2=0\},
\end{equation*}
we define a germ of a Lagrangian submanifold
$(L,\lambda_0)\subset T^*\R^n$ as the image of the
map
\begin{equation*}
C\ni (u,x)\mapsto \sum_{j=1}^n
\frac{\partial F}{\partial x_j}(u,x)dx_j\in T^*\R^n\,,\quad \lambda_0=\sum_{j=1}^n\frac{\partial F}{\partial x_j}(0,0)dx_j\,.
\end{equation*}
Namely, $L\subset T^*\R^n=\{(x,\xi)\}$ is parametrized by $(u_1,u_2,x_3.\dots,x_n)$ as
\begin{equation*}
x_1=-(3u_1^2+u_2^2),\quad x_2=-2(u_1+x_3)u_2, 
\quad\xi=(u_1,u_2,u_2^2,c_4,\dots,c_n)\,.
\end{equation*}
The Lagrangian equivalence class represented by
\begin{equation*}
\pi\circ i: (L,\lambda_0)\to (\R^n,0)
\end{equation*}
is called the $D_4^+$ Lagrangian singularity.
\subsection{Singularities at points with double
conjugacy}
We now come back to the situation at \S6.
Let $p_0=x_0=(x_{1,0},\dots,x_{n,0})\in M$ be a general point and let $\lambda_0=(x_0,\xi_0)\in U^*_{p_0}M$ be a covector
where $b_j=b_{j-1}$ for a fixed $j$ ($2\le j\le n-1$), i.e., $\lambda_0\in C_j^+\cap C_{j-1}^-$. We shall denote by $b_{k,0}$ the value of $b_k$ at $\lambda_0$ $(1\le k\le n-1)$. Since the coordinate functions $u_{j-1}$ and $u_j$ on $U^*_{p_0}M$ in the previous section are not appropriate at $\lambda_0$,
we introduce the following functions $\nu_1$, $\nu_2$ instead:
\begin{equation*}
2\nu_1=b_j+b_{j-1}-2f_{j,0},\quad \nu_2=\epsilon\sqrt{(b_{j-1}-f_{j,0})(f_{j,0}-b_j)}\,,
\end{equation*}
where $f_{j,0}=f_j(x_{j,0})$ and $\epsilon=\pm 1$ is chosen so that it is the sign of
$\xi_j$. Thus $\nu_1$ and $\nu_2$ are smooth functions on $U_{p_0}^*M$ around $\lambda_0$, $d\nu_1\wedge
d\nu_2\ne 0$, and
\begin{gather*}
\xi_j=\nu_2\sqrt{(-1)^{j}\prod_{k\ne j,j-1}(f_{j,0}-b_k)}\,,\\
\xi_i=\epsilon_i\sqrt{(-1)^{i-1}\prod_{k\ne j,j-1}(f_{i,0}-b_k)}\\
\times\sqrt{(f_{i,0}-f_{j,0})^2-2(f_{i,0}-f_{j,0})\nu_1- \nu_2^2}\quad (i\ne j)\,,
\end{gather*}
where $\epsilon_i$ is the same one as in \S6 and $f_{i,0}=f_i(x_{i,0})$.

Also we take coordinate functions $(\tilde w_1,\dots,\tilde w_{n-3})$ instead of $b_k$'s $(k\ne j,j-1)$ so that the product structure
\begin{equation*}
\{d\nu_1=d\nu_2=0\}\times\{d\tilde w_k=0, 1\le k\le n-3\}
\end{equation*}
coincides with that of
\begin{equation*}
\{db_{j}=db_{j-1}=0\}\times\{db_k=0, 1\le k\le n-1, k\ne j,j-1\}\,.
\end{equation*}
(One can take them as $u_k$'s $(k\ne j,j-1)$ if $\lambda_0$ is not
contained in any $\partial C_k^\pm$ other than $\partial C_j^+=\partial C_{j-1}^-$.)
We put
\begin{equation*}
\tilde S=\{\lambda\in W\subset U_{p_0}^*M\, |\,\nu_1(\lambda)=\nu_2(\lambda)=0\}\,,
\end{equation*}
where $W$ is a neighborhood of $\lambda_0$ in $U^*_{p_0}M$. We shall use the abbreviated notations
\begin{equation*}
\nu=(\nu_1,\nu_2),\quad \tilde w=(\tilde w_1,\dots,\tilde w_{n-3}),\quad \lambda=(\nu,\tilde w)
\in U^*_{p_0}M\,.
\end{equation*}
We take $\tilde w$ so that $\lambda_0=(0,0)$.

Let $\gamma(t)=\gamma(t,\nu,\tilde w)=(x_1(t,\nu,\tilde w),\dots, x_n(t,\nu,\tilde w))$ be the geodesic
such that $\gamma(0)=p_0$ and $\flat(\dot\gamma(0))=(\nu,\tilde w)$. 
Let $Z_{j-1}(t)$ and $Z_j(t)$ be the Jacobi fields
along the geodesic $\gamma(t)$ defined by the initial conditions
\begin{equation*}
Z_{j-1}(0)=0,\ Z_j(0)=0,\quad Z'_{j-1}(0)=
\frac12\,\sharp\left(\frac{\partial}{\partial \nu_1}\right),\ 
Z'_j(0)=\frac12\,\sharp\left(\frac{\partial}{\partial \nu_2}\right)\,.
\end{equation*}
They are equal with the Jacobi fields $Z_{j-1,0}(t)$ and $Z_{j,0}(t)$ given in p.271 of our
previous paper \cite[\S5]{IK2}.
There we proved the following proposition (\cite[p.272]{IK2}), which we also need here.
Let $t=\tau_1>0$ be the first zero of $Z_j(t)$ along the geodesic $\gamma(t,0,0)$.%
\begin{prop}
\begin{enumerate}
\item $Z_{j-1}(\tau_1)=Z_j(\tau_1)=0$\,.
\item $Z_{j-1}(t)$ and $Z_j(t)$ are linearly independent for any $0<t<\tau_1$.
\end{enumerate}
\end{prop}
We now assume that the following condition is satisfied:
\begin{equation}\label{assumption}
\begin{aligned}
&\text{\it  There is no
Jacobi field $Y(t)\not\equiv 0$ with $Y(0)=0, Y(\tau_1)=0$} \\
&\text{\it along the geodesic $\gamma(t,0,0)$
other than linear combinations}\\
&\text{\it  of $Z_j(t)$ and $Z_{j-1}(t)$.}
\end{aligned}
\end{equation}
This condition is automatically satisfied when $j=n-1$ by Proposition~\ref{prop:conj}.
Put 
\begin{equation*}
\tilde L=\{t\lambda\,|\, |t-\tau_1|<\epsilon,\ \lambda\in W\subset U^*_{p_0}M\}\subset T^*_{p_0}M
\end{equation*}
for a small constant $\epsilon>0$ and let $\phi:\tilde L\to T^*M$ by
\begin{equation*}
\phi(t\lambda)=\zeta_1(t\lambda)=t\zeta_t(\lambda)\,,
\end{equation*}
where $\{\zeta_t\}$ denotes the geodesic flow on $T^*M$.
Put 
\begin{equation*}
L=\phi(\tilde L),\quad \lambda_1=\phi(\tau_1\lambda_0)\,.
\end{equation*}
Then $L$ is a Lagrangian submanifold of $T^*M$, and we have the following
\begin{thm}\label{thm:d4plus}
The map-germ $\pi|_L:(L,\lambda_1)\to (M,p_1)$ is a $D_4^+$ Lagrangian singularity.
\end{thm}

To prove this theorem we shall prepare good coordinate functions $y_0,y_1,y_2$, and $w_k$ $(1\le k\le n-3)$
around the point $p_1=\gamma(\tau_1,0,0)\in M$ so that the criterion in Theorem~\ref{criterion} will be
easily applicable. First, we define $y_0,y_1$ and $y_2$. If the condition
\begin{equation}\label{req2}
f_l(x_l(\tau_1,0,0))\ne b_{k,0}\qquad \text{for any }k\ne j-1, j\,\text{and }1\le l\le n
\end{equation}
is satisfied, then we put:
\begin{align*}
y_0=&A_1(f_{j,0})(f_j(x_j)-f_{j,0 })\\
y_\alpha=&\sum_{1\le i\le n\atop{i\ne j}}
\int_{x_{i,1}}^{x_i}\epsilon_i\sqrt{(-1)^{i-1}
\prod_{k\ne j-1,j}(f_i(x_i)-b_{k,0})}
\,\frac{(f_i(x_i)-f_{j,0})^{\alpha}}{|f_i(x_i)-f_{j,0}|}\,dx_i\\ 
&\hfil (\alpha=1,2)\,,
\end{align*}
where $x_{i,1}=x_i(\tau_1,0,0)$, $\epsilon_i$ is the sign of $(\partial x_i/\partial t)(\tau_1,0,0)$, and
\begin{equation}\label{a1def}
A_1(\lambda)=\frac{\sqrt{(-1)^j\prod_{k\ne j,j-1}(\lambda-b_k)}\,A(\lambda)}{2\sqrt{(-1)^j\prod_{l=0}^n(\lambda-a_l)}}\,.
\end{equation}

If \eqref{req2} is not satisfied for some $k$, then we put
\begin{equation*}
I=\{i\, |\, 1\le i\le n,\, i\ne j,\, f_i(x_{i,1})\ne b_{l,0}
\text{ for any }l\ne j,j-1\}
\end{equation*}
and
\begin{align*}
y_\alpha=\sum_{i\in I}
\int_{x_{i,1}}^{x_i}\epsilon_i\sqrt{(-1)^{i-1}
\prod_{k\ne j-1,j}(f_i(x_i)-b_{k,0})}
\,\frac{(f_i(x_i)-f_{j,0})^{\alpha}}{|f_i(x_i)-f_{j,0}|}\,dx_i\\ 
\hfil (\alpha=1,2)\,.
\end{align*}
Next, we shall define $(w_1,\dots,w_{n-3})$.
First we define them on the submanifold 
\begin{equation*}
S=\{\gamma(t,0,\tilde w)\,|\, (0,\tilde w)\in \tilde S\subset U^*_{p_0}M,\, |t-\tau_1|<\epsilon\}
\end{equation*}
by $w_k(\gamma(t,0,\tilde w))=\tilde w_k$ $(1\le k\le n-3)$. Note that $S$ is really a submanifold due to the  assumption~\eqref{assumption}. Along $S$ we define mutually orthogonal unit vector fields $V_1$ and $V_2$ which are normal to $S$. Then we extend $w_k$'s
to a neighborhood of $p_1$ in $M$ so that they
satisfy
\begin{equation*}
dw_k(V_i)=0\qquad (1\le k\le n-3,\,i=1,2)
\end{equation*}
at each point on $S$.
\begin{lemma}
$dy_0\wedge dy_1\ne 0$ at $p_1$,
$dy_0=dy_1=0$ on $T_{p_1}S$, and 
\begin{equation*}\frac{d}{dt}y_2(\gamma(t,0,0))\ne 0\,. 
\end{equation*}
In particular, $(y_0,y_1,y_2,w_1,\dots,w_{n-3})$
form a coordinate system of $M$ around $p_1$.
\end{lemma}
\begin{proof}
$dy_0\wedge dy_1\ne 0$ is obvious, since $y_1$
does not contain the variable $x_j$. Since $f_j(x_j)$ remains constant ($=f_{j,0}$) on the
geodesic $\gamma(t,0,\tilde w)$, we have $dy_0|_{TS}=0$ and $dy_0\ne 0$ at each point on $S$.
For $y_1$, we observe Proposition~\ref{prop:dualtob}, which is effective for $\lambda\in U^*_{p_1}M$
such that $b_1,\dots,b_{n-1}$ and $a_1,\dots,a_n$ are all distinct. We then have
\begin{equation*}
\lim_{\lambda\to\lambda_1}(\tilde \omega_j+\tilde \omega_{j-1})=2dy_1\,,
\end{equation*}
and it therefore follows that $dy_1=0$ on $TS$. Also, for $y_2$, we observe the formula \eqref{eq:geodlength}.
Taking a limit as above, we have 
\begin{equation*}
\frac{d}{dt}y_2(\gamma(t,0,0))\ne 0\,.
\end{equation*}
\end{proof}

Let $(\eta_0,\eta_1,\eta_2,v_1,\dots,v_{n-3})$ be the canonical fiber coordinates of $T^*M$
associated with the coordinate system $(y_0,y_1,y_2,w_1,\dots,w_{n-3})$ of $M$.
Using the coordinate system $(t,\nu,\tilde w)$ on $\tilde L$, we put
\begin{gather*}
y_\alpha\circ\phi(t,\nu,\tilde w)=y_\alpha(t,\nu,\tilde w)\quad  (0\le \alpha\le 2),\\
w_k\circ\phi(t,\nu,\tilde w)=w_k(t,\nu,\tilde w)\quad  (1\le k\le n-3),
\end{gather*}
and we also define $\eta_\alpha(t,\nu,\tilde w)$ and $v_k(t,\nu,\tilde w)$ in the same way.
Note that
\begin{equation*}
\pi\circ\phi(t,\nu,\tilde w)=\gamma(t,\nu,\tilde w).
\end{equation*}
Therefore we have $y_\alpha(t,\nu,\tilde w)=y_\alpha(\gamma(t,\nu,\tilde w))$, etc..
For those functions we have the following proposition; the proof 
will be given in the next subsection.
\begin{prop}\label{mainprop}
There are nonzero constants $c$ and $c'$ such that:
\begin{align*}
y_0(t,\nu,0)= & \,2c\nu_1\nu_2+c'\nu_2(t-\tau_1)+O((|\nu|+|t-\tau_1|)^3)\\
y_1(t,\nu,0)= & \,c(3\nu_1^2+\nu_2^2)+c'\nu_1(t-\tau_1)+O((|\nu|+|t-\tau_1|)^3)\\
y_2(t,\nu,0)= & \,t-\tau_1+O((|\nu|+|t-\tau_1|)^3)\\
w_k(t,\nu,0)= & \,O((|\nu|+|t-\tau_1|)^3)\,.
\end{align*}
\end{prop}

By this proposition we have the following lemmas.
\begin{lemma}
\begin{enumerate}
\item $\dim\ker (\pi|_L)_*=2$ at $\lambda_1$.
\item The system of functions $(\eta_0,\eta_1,y_2,w_1,\dots,w_{n-3})$ becomes a coordinate system of $L$
around the point $\lambda_1$.
\end{enumerate}
\end{lemma}
\begin{proof}
Proposition~\ref{mainprop} and the fact that $\partial w_k/\partial \tilde w_l=\delta_{kl}$
at $p_1$ imply that $(\pi|_L)^*(dy_0)=(\pi|_L)^*(dy_1)=0$ at $\lambda_1$ and
\begin{equation*}
(\pi|_L)^*(dy_2\wedge dw_1\wedge\dots\wedge dw_{n-3})\ne 0 \text{ at }\lambda_1\,.
\end{equation*}
Since $L$ is a Lagrangian submanifold, the lemma follows from those facts.
\end{proof}
\begin{lemma}\label{lemeta}
\begin{align*}
\eta_0(\tau_1,\nu,0)=& \eta_0(\tau_1,0,0)+e\nu_2+O(|\nu|^2)\,,\\
\eta_1(\tau_1,\nu,0)=& \eta_1(\tau_1,0,0)+e\nu_1+O(|\nu|^2)\,
\end{align*}
for some nonzero constant $e$.
\end{lemma}
\begin{proof}
Since $L$ is Lagrangian, it follows that
\begin{equation*}
\phi^*\left(\sum_{\alpha=0}^2 d\eta_\alpha\wedge dy_\alpha+\sum_{k=1}^{n-3}dv_k\wedge dw_k
\right)=0\,.
\end{equation*}
Taking the coefficients of $d\nu_1\wedge d\nu_2$ of the left-hand side, one therefore obtains, by Proposition~\ref{mainprop},
\begin{gather*}
a^0_1 (2c\nu_1+c'(t-\tau_1))-a^0_2\cdot 2c\nu_2+a^1_1\cdot 2c\nu_2\\
-a^1_2(6c\nu_1+c'(t-\tau_1))+O((|\nu|+|t-\tau_1|)^2)=0\,,
\end{gather*}
where $a^\alpha_i=\partial \eta_\alpha(t,\nu,0)/\partial \nu_i$.
This implies that
\begin{equation*}
a^0_1=3a^1_2,\quad a^0_2=a^1_1,\quad a^0_1=a^1_2\quad\text{at }\nu=0,\, t=\tau_1.
\end{equation*}
Therefore $a^0_1=a^1_2=0$ there and we obtain the desired formula. Since $d\eta_0\ne 0$ at $\lambda_1$,
we also have $e\ne 0$.
\end{proof}

We now define the function $\hat F(\eta_0,\eta_1,y_2,w_1,\dots,w_{n-3})$ on $L$ as an integral of the
closed form 
\begin{equation*}
-y_0d\eta_0-y_1d\eta_1+\hat\eta_2dy_2+\sum_{k=1}^{n-3}\hat v_kdw_k =\alpha-d(\eta_0y_0+\eta_1y_1)-dh\,,
\end{equation*}
where $\alpha$ denotes the canonical $1$-form,
\begin{equation*}
\hat \eta_\alpha=\eta_\alpha-\eta_\alpha(\tau_1,0,0),\quad \hat v_k=v_k-v_k(\tau_1,0,0)
\end{equation*}
for $0\le\alpha\le 2$ and $1\le k\le n-3$, and
\begin{equation*}
h=h(y_2,w_1,\dots,w_{n-3})=\eta_2(\tau_1,0,0)y_2+\sum_{k=1}^{n-3}v_k(\tau_1,0,0)w_k\,.
\end{equation*}

We may take $\hat F$ so that $\hat F=0$ at $\lambda_1\in L$.
Then, as stated in the previous subsection, 
\begin{equation*}
F(\eta_0,\eta_1,y_0,y_1,y_2,w_1,\dots,w_{n-3})=\eta_0y_0+\eta_1y_1+\hat F+h
\end{equation*}
becomes a generating family for $L$. (Note that
$F$ contains $(y_0,y_1)$ as independent variables, while $\hat F$ does not.) 
\begin{lemma}\label{final}
$F(\eta_0, \eta_1,0,\dots,0)=-ce^{-2}(\hat\eta_0^2\hat\eta_1+\hat\eta_1^3)+O(|\hat\eta|^4)$\,.
\end{lemma}
\begin{proof}
It is enough to show that $\hat F(\eta_0,\eta_1,0,\dots,0)$ is equal to the right-hand side.
First we have, by Proposition~\ref{mainprop} and Lemma~\ref{lemeta}, 
\begin{equation*}
\phi^*d\hat F(\tau_1,\nu,0)=-ce(2\nu_1\nu_2d\nu_2+(3\nu_1^2+\nu_2^2)d\nu_1)+O(|\nu|^3)d\nu\,,
\end{equation*}
and therefore
\begin{equation*}
\phi^*\hat F(\tau_1,\nu,0)=-ce(\nu_1^3+\nu_1\nu_2^2)+O(|\nu|^4)\,.
\end{equation*}
We then need to evaluate the difference
\begin{equation*}
\hat F(\eta_0,\eta_1,y_2,w_1,\dots,w_{n-3})-\hat F(\eta_0,\eta_1,0,\dots,0),
\end{equation*}
which is equal to $Ay_2+\sum_{k=1}^{n-3}B_kw_k$, where
\begin{gather*}
A=A(\eta_0,\eta_1,y_2,w_1,\dots,w_{n-3})=\int_0^1\hat\eta_2(\eta_0,\eta_1,sy_2,sw_1,\dots,sw_{n-3})ds,\\
B_k=B_k(\eta_0,\eta_1,y_2,w_1,\dots,w_{n-3})=\int_0^1
\hat v_k(\eta_0,\eta_1,sy_2,sw_1,\dots,sw_{n-3})ds.
\end{gather*}
Let us pull back this formula by $\phi$ at $(\tau_1,\nu,0)$. Since $y_2(\tau_1,0,0)=w_k(\tau_1,0,0)=0$,
we have $\phi^*A(\tau_1,0,0)=\phi^*B_k(\tau_1,0,0)=0$. Therefore it follows that
\begin{equation*}
\phi^*(Ay_2+\sum_{k=1}^{n-3}B_kw_k)(\tau_1,\nu,0)+O(|\nu|^4)=0
\end{equation*}
and thus
\begin{equation*}
\phi^*(\hat F(\eta_0,\eta_1,0,\dots,0))(\tau_1,\nu,0)=-ce(\nu_1^3+\nu_1\nu_2^2)+O(|\nu|^4)\,.
\end{equation*}
The lemma then follows from Lemma~\ref{lemeta}.
\end{proof}

\begin{lemma}\label{eta2}
The function $\hat\eta_2$, restricted to the submanifold $L'$:
\begin{equation*}
L'=\{(\eta_0,\eta_1,y_2,w)\in L\,|\, y_2=w=0\}
\end{equation*}
is described as 
\begin{equation*}
\hat\eta_2|_{L'}=c_1(\hat\eta_0^2+\hat\eta_1^2)+c_2\hat\eta_0\hat\eta_1+c_3(\hat\eta_0^2+3\hat\eta_1^2)+O(|\hat\eta|^3)\,,
\end{equation*}
where $c_1,c_2,c_3$ are constants and $c_1\ne 0$.
\end{lemma}
\begin{proof}
We compute the coefficients of $d\nu_i\wedge dt$
in the $2$-form
\begin{equation*}
\phi^*\left(\sum_{\alpha=0}^2 d\eta_\alpha\wedge dy_\alpha+\sum_{k=1}^{n-3}dv_k\wedge dw_k
\right)=0\,.
\end{equation*}
at the points $(\tau_1,\nu,0)$.
By Proposition~\ref{mainprop} and Lemma~\ref{lemeta}, we have
\begin{equation*}
-2c\nu_2\frac{\partial \eta_0}{\partial t}-6c\nu_1
\frac{\partial \eta_1}{\partial t}+ec'\nu_1+
\frac{\partial \eta_2}{\partial \nu_1}+O(|\nu|^2)=0
\end{equation*}
as the coefficients of $d\nu_1\wedge dt$, and
\begin{equation*}
-2c\nu_1\frac{\partial \eta_0}{\partial t}-2c\nu_2
\frac{\partial \eta_1}{\partial t}+ec'\nu_2+
\frac{\partial \eta_2}{\partial \nu_2}+O(|\nu|^2)=0
\end{equation*}
as the coefficients of $d\nu_2\wedge dt$. Therefore we obtain
\begin{equation*}
\hat\eta_2(\tau_1,\nu,0)=-\frac{ec'}2(\nu_1^2+\nu_2^2)+2c\nu_1\nu_2\frac{\partial\eta_0}{\partial t}+c(3\nu_1^2+\nu_2^2)\frac{\partial \eta_1}{\partial t}+O(|\nu|^3)\,.
\end{equation*}
We note that $ec'\ne 0$.

Here we need, as in the previous lemma, to observe
the difference $\hat\eta_2|_{L}-\hat\eta_2|_{L'}$.
Since it is described in the form
\begin{equation*}
Ay_2+\sum_{k=1}^{n-3}B_kw_k
\end{equation*}
for certain functions $A$ and $B_k$, as in the proof of the previous lemma, it follows that 
\begin{equation*}
\phi^*(\hat\eta_2|_{L}-\hat\eta_2|_{L'})(\tau_1,\nu,0)=O(|\nu|^3)\,.
\end{equation*}
Therefore we have 
\begin{equation*}
\phi^*(\hat\eta_2|_{L'})(\tau_1,\nu,0)=
-\frac{ec'}2(\nu_1^2+\nu_2^2)+2c\nu_1\nu_2\frac{\partial\eta_0}{\partial t}+c(3\nu_1^2+\nu_2^2)\frac{\partial \eta_1}{\partial t}+O(|\nu|^3)\,.
\end{equation*}
Thus the lemma follows by Lemma~\ref{lemeta}.
\end{proof}
We now prove Theorem~\ref{thm:d4plus}. By Lemma~\ref{final} the function germ 
$f=F(\eta_0,\eta_1,0,\dots,0)$
is equivalent to the $D^+_4$ function germ, since the latter is $3$-determined. 
Also, since $\partial F/\partial y_2=\hat\eta_2$,
we see by Lemma~\ref{eta2} that $F$ is a versal
deformation of $f$.
Therefore, applying the criterion of Theorem~\ref{criterion} to the generating family $F$ for $(L,\lambda_1)$,
we see that the map-germ $\pi|_L:(L,\lambda_1)\to(M,p_1)$ is a $D_4^+$ Lagrangian singularity.
This completes the proof of Theorem~\ref{thm:d4plus} under the assumption of Proposition~\ref{mainprop}.\qed

As direct consequences of Theorem~\ref{thm:d4plus}, we have the following corollaries.
\begin{cor}\label{cor1}
The germ of the map $\pi\circ \zeta_1: T^*_{p_0}M\to M$ at $\lambda\in \flat(\tilde K_{n-1})$ is a
$D_4^+$ Lagrangian singularity if $\lambda/|\lambda|\in\partial C_{n-1}^+$.
\end{cor}
\begin{cor}\label{cor2}
Suppose that the second zero $t=r^2_{n-1}(u)$ of the Jacobi field $Y_{n-1}(t,u)$ is greater than $r_1(u)$
for any $[u]\in U^*_{p_0}M$.
Then the germ of the map $\pi\circ\zeta_1: T^*_{p_0}M\to M$ at $\lambda\in \flat(\tilde K_i)$ is a 
$D_4^+$ Lagrangian singularity if $\lambda/|\lambda|\in \partial C_i^+\cup \partial C_i^-$ $(1\le i\le n-2)$.
\end{cor}

\subsection{Proof of Proposition~\ref{mainprop}}
In this subsection we shall always assume $\tilde w=0$, so we shall shortly write $(t,\nu)$ instead of writing $(t,\nu,0)$. 
Also, the value of $b_k$ $(k\ne j,j-1)$ will be fixed to be $b_{k,0}$ throughout this subsection.

First, we would like to define a function 
$\theta(t,\nu)$ satisfying
\begin{equation}\label{theta}
f_j(x_j(t,\nu))-f_{j,0}=\nu_1(1-\cos\theta(t,\nu))+\nu_2\sin\theta(t,\nu)\,.
\end{equation}
\begin{lemma}\label{lemma:theta}
There is a unique $C^\infty$ function $\theta(t,\nu)$ for small $|\nu|$ and $t\in\R$ satisfying \eqref{theta} and the initial condition $\theta(0,\nu)=0$. Moreover, it satisfies
\begin{equation*}
\frac{\partial\theta}{\partial t}(t,\nu)>0\qquad \text{for any } (t,\nu)\,.
\end{equation*}
\end{lemma}
\begin{proof}
The formula \eqref{theta} is equivalent to 
\begin{equation*}
f_j(x_j(t,\nu))=b_{j-1}\left(\cos ((\theta+\alpha)/2)\right)^2
+b_{j}\left(\sin ((\theta+\alpha)/2)\right)^2\,,
\end{equation*}
where $\theta=\theta(t,\nu)$ and $\alpha$ is 
defined by
\begin{equation*}
-(\nu_1,\nu_2)=\sqrt{\nu_1^2+\nu_2^2}\,(\cos\alpha,\sin\alpha)\,.
\end{equation*}
If $\nu_2\ne 0$, then $b_{j-1}>f_j(x_j(0,\nu))>b_j$ and the function $t\mapsto f_j(x_j(t,\nu))$
oscillates between $b_j$ and $b_{j-1}$ and the second derivaties do not vanish at the turning points.
Therefore the assertion easily follows in this case, and we have $\partial\theta/\partial t\ne 0$ for any $t$.
Since 
\begin{equation*}
\frac{d}{dt}f_j(x_j(t,\nu))\vert_{t=0}=\nu_2\frac{\partial\theta}{\partial t}(0,\nu)\,,
\end{equation*}
and since
\begin{equation*}
\text{sign of }\nu_2= \text{ sign of }\xi_j= \text{ sign of } df_j(x_j)/dt \text{ at }t=0\,, 
\end{equation*}
it follows that $\partial\theta/\partial t>0$ at $t=0$ and so for any $t$. (Note that $\partial f_j/\partial x_j>0$ by the assumption posed at the beginning of \S6.) 

Now let us verify that the function $\theta(t,\nu)$ thus obtained for $\nu_2\ne 0$ is smoothly
extended to points where $\nu_2=0$.
Putting $G(\lambda)=\prod_{k\ne j,j-1}(\lambda-b_k)$ in the formula \eqref{geodeqint}, we have
\begin{equation*}
\sum_{i=1}^n U_i(t,\nu)=0\,,
\end{equation*}
where $U_i(t,\nu)$ is given by
\begin{equation*}
\int_0^{t}\frac{(-1)^i\prod_{k\ne j,j-1}(f_i-b_k)\,|{\partial x_i(t,\nu)}/{\partial t}|\ dt}
{\sqrt{(-1)^{i-1}\prod_{k\ne j,j-1}(f_i-b_k)
\left((f_i-f_{j,0})^2-2\nu_1(f_i-f_{j,0})-\nu_2^2\right)}}\,.
\end{equation*}
When $\nu_2\ne 0$, $U_j$ is rewritten as
\begin{equation*}
U_j(t,\nu)=\int_0^{\theta(t,\nu)}
A_1(f_j)\,d\theta\,,
\end{equation*}
where $f_j=f_{j,0}+\nu_1(1-\cos\theta)+\nu_2\sin\theta$ and $A_1(\lambda)$ is as in \eqref{a1def}.
This formula redefine $\theta(t,\nu)$, which is effective for the case $\nu_2=0$ and is of $C^\infty$
anywhere. Here, we again have $\partial\theta/\partial t>0$, since
\begin{equation*}
(-1)^i\prod_{k\ne j,j-1}(f_i(x_i(t,\nu))-b_k)
\,|{\partial x_i(t,\nu)}/{\partial t}|\le 0
\end{equation*}
for any $i\ne j,j-1$ and is strictly negative for
$i$ with $L_i$ being the whole circle.
\end{proof}
It should be noted that when $\nu=0$ the function
$f_j(x_j(t,\nu))$ is identically equal to 
$f_{j,0}$ (constant), but the function $\theta(t,0)$ is strictly increasing in $t$.  
\begin{lemma}
$\theta(\tau_1,0)=2\pi.$
\end{lemma}
\begin{proof}
When $\nu\ne 0$, we have
\begin{equation*}
2(b_{j-1}-b_j)=\sigma_j(t_j)=\int_0^{t_j}\left|\frac{df_j(x_j(s,\nu))}{ds}\right|ds\,.
\end{equation*}
The right-hand side is equal to 
\begin{gather*}
\int_0^{t_j}\left|\nu_1\sin\theta(s,\nu)+\nu_2\cos\theta(s,\nu)\right|\frac{\partial\theta}{\partial s}ds\\
=\sqrt{\nu_1^2+\nu_2^2}\int_0^{\theta(t,\nu)}
|\sin(\theta+\alpha)|d\theta\,,
\end{gather*}
where $\alpha$ is the same one as in the proof of 
Lemma~\ref{lemma:theta}.
Since $b_{j-1}-b_j=2\sqrt{\nu_1^2+\nu_2^2}$, it 
therefore follows that $\theta(t_j,\nu)=2\pi$.

On the other hand, by Propositions~\ref{prop:main} and \ref{prop:conj},
we have $r_j\le t_j\le r_{j-1}$ when
$\nu\ne 0$, and $r_{j-1},r_j$ tend to $\tau_1$ when $\nu$ tends to $0$. Therefore we have $\theta(\tau_1,0)=2\pi$ by continuity.
\end{proof}
We now consider the geodesic equations \eqref{geodeqint} and \eqref{eq:geodlength} for the following polynomials $G(\lambda)$:
\begin{equation*}
G(\lambda)=\prod_{k\ne j,j-1}(\lambda-b_k)\cdot
(\lambda-f_{j,0})^\alpha\qquad(\alpha=0,1,2)\,.
\end{equation*}
Let us put, for each $\alpha=0,1,2$:
\begin{gather*}
u_i^\alpha(x_i,\nu)=\frac{-\sqrt{(-1)^{i-1}\prod_{k\ne j,j-1}(f_i-b_{k})}\,(f_i-f_{j.0})^\alpha}
{\sqrt{(f_i-f_{j,0})^2-2\nu_1(f_i-f_{j,0})-\nu_2^2}}\qquad (i\ne j)\,,\\
u_j^\alpha(\theta,\nu)=A_1(f_j)\,(f_j-f_{j,0})^\alpha\,,
\end{gather*}
and
\begin{gather*}
U_i^\alpha(t,\nu)=\int_0^tu_i^\alpha(x_i(t,\nu),\nu)\left|\frac{\partial x_i}{\partial t}(t,\nu)\right|dt
\quad(i\ne j)\,,\\
U_j^\alpha(t,\nu)=\int_0^{\theta(t,\nu)}u_j^\alpha(\theta,\nu)\,d\theta\,,
\end{gather*}
where $f_i=f_i(x_i)$ $(i\ne j)$ and
\begin{equation*}
f_j=f_{j,0}+\nu_1(1-\cos\theta)+\nu_2\sin\theta\,.
\end{equation*}
We then have
\begin{equation}\label{large-u}
\sum_{1\le i\le n\atop{i\ne j}}U_i^\alpha(t,\nu)
+U_j^\alpha(t,\nu)=\begin{cases}
0\quad(\alpha=0,1)\\
t\quad(\alpha=2)
\end{cases}.
\end{equation}
The functions $y_0$, $y_1$, and $y_2$ clearly have
the following relations with the above functions: 
\begin{align}
y_0=&A_1(f_{j,0})(f_j(x_j)-f_{j,0})\,,\\
dy_\alpha=&\sum_{i\in I}\epsilon_i u_i^\alpha(x_i,0)dx_i\quad (\alpha=1,2)\,.
\label{eq:yalpha}\end{align}
Therefore for the proof of the proposition it
is necessary to calculate the first and second
derivatives of the above functions at $(t,\nu)=(\tau_1,0)$.

Let us start with the derivatives in $t$:
\begin{align}\label{der-t1}
\frac{\partial U_i^\alpha}{\partial t}(t,\nu)=&
\epsilon_iu_i^\alpha(x_i(t,\nu),\nu)\frac{\partial x_i}{\partial t}\quad (i\ne j)\\
\frac{\partial U_j^\alpha}{\partial t}(t,\nu)=&
\frac{\partial \theta}{\partial t}(t,\nu)\,u_j^\alpha(\theta(t,\nu),\nu)\,.
\label{der-t2}\end{align}
Taking the fact
\begin{equation}
U_j^\alpha(t,0)=0\quad\text{for any }t\in\R\text{ and }\alpha=1,2
\end{equation}
into account, we have
\begin{equation*}
\sum_{i\ne j}\frac{\partial U_i^\alpha}{\partial t}(t,0)=\begin{cases}
0\quad(\alpha=1)\\
1\quad(\alpha=2)
\end{cases}.
\end{equation*}
Since $u_i^\alpha(x_{i,1},0)=(\partial x_i/\partial t)(\tau_1,0)=0$ if $f_i(x_{i,1})=b_k$ for some $k\ne j,j-1$, it therefore follows that
\begin{equation}\label{der-t12}
\frac{\partial y_\alpha}{\partial t}(\tau_1,0)
=\begin{cases}
0\quad(\alpha=1)\\
1\quad(\alpha=2)
\end{cases},\quad
\frac{\partial^2 y_\alpha}{\partial t^2}(\tau_1,0)=0\quad(\alpha=1,2)\,.
\end{equation}
We also have
\begin{equation}\label{y0t}
y_0(t,0)=0\quad \text{for any }t\in\R\,.
\end{equation}

Next, let us consider the derivatives
$\partial U_i^\alpha/\partial t\partial \nu_k$.
We have from \eqref{der-t1} and \eqref{der-t2}:
\begin{gather*}
\frac{\partial U_i^\alpha}{\partial t\partial \nu_k}(t,\nu)=\epsilon_i\frac{\partial u_i^\alpha}{\partial x_i}(x_i(t,\nu),\nu)\frac{\partial x_i}{\partial \nu_k}\frac{\partial x_i}{\partial t}\\
+\epsilon_i\frac{\partial u_i^\alpha}{\partial \nu_k}(x_i(t,\nu),\nu)\frac{\partial x_i}{\partial t}+\epsilon_iu_i^\alpha(x_i(t,\nu),\nu)\frac{\partial^2x_i}{\partial t\partial \nu_k}
\end{gather*}
and
\begin{gather*}
\frac{\partial U_j^\alpha}{\partial t\partial \nu_k}(t,\nu)=\frac{\partial^2\theta}{\partial t\partial\nu_k}(t,\nu)\,u_j^\alpha(\theta,,\nu)+\frac{\partial\theta}{\partial t}\frac{\partial \theta}{\partial\nu_k}\frac{\partial u_j^\alpha}{\partial \theta}(\theta,\nu)\\
+\frac{\partial\theta}{\partial t}\frac{\partial u_j^\alpha}{\partial \nu_k}(\theta,\nu)\,.
\end{gather*}
Since $\partial x_i/\partial \nu_k$ $(k=1,2)$ is the coefficient of $\partial/\partial x_i$ in the
Jacobi field $Z_j$ or $Z_{j-1}$, it vanishes at
$(t,\nu)=(\tau_1,0)$. Also, we have
\begin{equation*}
u_j^\alpha(2\pi,0)=\frac{\partial u_j^\alpha}{\partial \theta}(2\pi,0)=\frac{\partial u_j^\alpha}{\partial \nu_k}(2\pi,0)=0
\end{equation*}
for $\alpha=1,2$. Therefore we have
\begin{equation}\label{der-tnu}
\sum_{i\ne j}\epsilon_i\frac{\partial u_i^\alpha}{\partial \nu_k}(x_{i,1},0)\frac{\partial x_i}{\partial t}(\tau_1,0)+\sum_{i\ne j}\epsilon_iu_i^\alpha(x_{i,1},0)\frac{\partial^2x_i}{\partial t\partial \nu_k}(\tau_1,0)=0
\end{equation}
for $\alpha=1,2$. Note that the second sum in the left-hand side of
the above equality is equal to $(\partial y_{\alpha}/\partial t\partial \nu_k)(\tau_1,0)$ and
\begin{equation*}
\frac{\partial u_i^\alpha}{\partial \nu_k}(x_{i,1},0)=\begin{cases}
u_{i}^{\alpha-1}(x_{i,1},0)\quad &(k=1)\\
0\quad &(k=2)\end{cases}\,.
\end{equation*}
Therefore,
\begin{equation}\label{der-tnu2}
\frac{\partial^2 y_\alpha}{\partial t\partial\nu_2}(\tau_1,0)=0\qquad (\alpha=1,2)
\end{equation}
and
\begin{equation}\label{der-tnu1}
\begin{gathered}
\frac{\partial^2 y_\alpha}{\partial t\partial\nu_1}(\tau_1,0)=-\sum_{i\ne j}u_i^{\alpha-1}(x_{i,1},0)\frac{\partial x_i}{\partial t}(\tau_1,0)
=-\sum_{i\ne j}\frac{\partial U_i^{\alpha-1}}{\partial t}(\tau_1,0)\\
=\frac{\partial U_j^{\alpha-1}}{\partial t}(\tau_1,0)=\begin{cases}
A_1(f_{j,0})\frac{\partial\theta}{\partial t}(\tau_1,0)\quad &(\alpha=1)\\
0\quad &(\alpha=2)
\end{cases}\,.
\end{gathered}
\end{equation}

For $y_0$, we have
\begin{equation*}
\frac{\partial y_0}{\partial t}(t,\nu)=A_1(f_{j,0})
(\nu_1\sin\theta+\nu_2\cos\theta)\frac{\partial \theta}{\partial t}\,.
\end{equation*}
Thus,
\begin{equation}\label{der-y0tnu}
\frac{\partial^2 y_0}{\partial t\partial \nu_k}
(\tau_1,0)=\begin{cases}
0\quad &(k=1)\\
A_1(f_{j,0})\frac{\partial \theta}{\partial t}(\tau_1,0)\quad &(k=2)
\end{cases}\,.
\end{equation}

Next, we shall consider the derivatives in $\nu_1$
and $\nu_2$.
For the first derivatives we have:
\begin{lemma}\label{lem:1stdernu}
$\frac{\partial y_\alpha}{\partial \nu_k}(\tau_1,0)=0$ for any $\alpha=0,1,2$ and $k=1,2$.
\end{lemma}
\begin{proof}
For the case $\alpha=0$ the assertion follows from
the fact that $\partial f_j/\partial \nu_k=0$ when
$\theta=2\pi$. For $\alpha=1,2$, we have
\begin{equation*}
\frac{\partial y_\alpha}{\partial \nu_k}(\tau_1,0)
=\sum_{i\in I}\epsilon_iu_i^\alpha(x_{i,1},0)\frac{\partial x_i}{\partial \nu_k}(\tau_1,0)\,.
\end{equation*}
Since $\partial x_i/\partial \nu_k$ is a component of the Jacobi fields $Z_j, Z_{j-1}$, it vanishes at
$(\tau_1,0)$. Thus the assertion follows.
\end{proof}

To compute the second derivatives in $\nu_1,\nu_2$,
we begin with the following lemma.
\begin{lemma}\label{der-unu} For each $\alpha=0,1,2$ and $k=1,2$,
\begin{equation*}
\frac{\partial U_i^\alpha}{\partial \nu_k}(t,\nu)
=\epsilon_i u_i^\alpha(x_i(t,\nu),\nu)\frac{\partial x_i}{\partial \nu_k}+
\int_0^t\frac{\partial u_i^\alpha}{\partial \nu_k}
\left|\frac{\partial x_i}{\partial t}\right|dt\,.
\end{equation*}
\end{lemma}
\begin{proof}
We have
\begin{equation*}
\frac{\partial U_i^\alpha}{\partial \nu_k}(t,\nu)
=\int_0^t\epsilon_i(t,\nu)\left[\left(\frac{\partial u_i^\alpha}{\partial x_i}\frac{\partial x_i}{\partial \nu_k}+\frac{\partial u_i^\alpha}{\partial \nu_k}
\right)\frac{\partial x_i}{\partial t}+u_i^\alpha
\frac{\partial^2 x_i}{\partial t\partial \nu_k}\right]dt\,.
\end{equation*}
Here $\epsilon_i(t,\nu)\quad(=\pm 1)$ stands for the sign of 
$(\partial x_i/\partial t)(t,\nu)$, which is locally constant in $t$ for each fixed $\nu$ outside the turning point, i.e., the point $t$
where $(\partial x_i/\partial t)(t,\nu)=0$.

Observe that the integrand is equal to
\begin{equation*}
\epsilon_i\frac{\partial}{\partial t}\left(
\frac{\partial x_i}{\partial \nu_k}u_i^\alpha\right)+\frac{\partial u_i^\alpha}{\partial \nu_k}
\left|\frac{\partial x_i}{\partial t}\right|\,.
\end{equation*}
Since $u_i^\alpha(x_i(t,\nu),\nu)$ vanishes at
each turning point and $(\partial x_i/\partial \nu_k)(t,\nu)$ vanishes at $t=0$,
we therefore obtain
\begin{equation*}
\int_0^t\epsilon_i(t,\nu)\frac{\partial}{\partial t}\left(
\frac{\partial x_i}{\partial \nu_k}u_i^\alpha\right)dt=\epsilon_i\frac{\partial x_i}{\partial \nu_k}(t,\nu)u_i^\alpha(x_i(t,\nu),\nu)\,.
\end{equation*}
Thus the lemma follows.
\end{proof}
\begin{lemma}\label{der-thetanu}
\begin{equation*}
\frac{\partial \theta}{\partial \nu_1}(\tau_1,0)=
-A_1(f_{j,0})^{-1}C,\quad \frac{\partial\theta}{\partial \nu_2}(\tau_1,0)=0\,,
\end{equation*}
where 
\begin{equation*}
C=\sum_{i\ne j}\int_0^{\tau_1}\frac{u_i^0(x_i(t,0),0)}{f_i(x_i(t,0))-f_{j,0}}\left|\frac{\partial x_i}{\partial t}(t,0)\right|\,dt+2\pi A_1'(f_{j,0})\,.
\end{equation*}
\end{lemma}
\begin{proof}
We use the formula
\begin{equation*}
\sum_{i\ne j}\frac{\partial U_i^0}{\partial \nu_k}(\tau_1,0)+\frac{\partial U_j^0}{\partial \nu_k}(\tau_1,0)=0\,.
\end{equation*}
By Lemma~\ref{der-unu} we have
\begin{equation*}
\frac{\partial U_i^0}{\partial \nu_k}(\tau_1,0)
=\begin{cases}
\int_0^{\tau_1}\frac{u_i^0(x_i(t,0),0)}{f_i(x_i(t,0))-f_{j,0}}\left|\frac{\partial x_i}{\partial t}\right|\,dt\quad &(k=1)\\
0\quad &(k=2)
\end{cases}\,.
\end{equation*}
Also, we have
\begin{equation*}
\frac{\partial U_j^0}{\partial \nu_k}(\tau_1,0)=
\frac{\partial\theta}{\partial\nu_k}(\tau_1,0)
A_1(f_{j,0})+\int_0^{2\pi}A_1'(f_{j,0})\frac{\partial f_j}{\partial \nu_k}\,d\theta
\end{equation*}
and
\begin{equation*}
\frac{\partial f_j}{\partial\nu_k}=
\begin{cases} 1-\cos\theta\quad &(k=1)\\
\sin\theta\quad &(k=2)
\end{cases}\,.
\end{equation*}
Therefore the lemma follows.
\end{proof}
\begin{cor}\label{cor:y0der2}
\begin{equation*}
\frac{\partial^2 y_0}{\partial \nu_1^2}(\tau_1,0)=
\frac{\partial^2 y_0}{\partial \nu_2^2}(\tau_1,0)=
0,\quad \frac{\partial^2 y_0}{\partial \nu_1 \partial \nu_2}(\tau_1,0)=-C\,.
\end{equation*}
\end{cor}
\begin{proof}
Since
\begin{equation*}
y_0(t,\nu)=A_1(f_{j,0})(\nu_1(1-\cos\theta(t,\nu))+\nu_2\sin\theta(t,\nu))\,,
\end{equation*}
the assertion easily follows from the previous lemma.
\end{proof}

Finally, we shall consider the second derivatives
of $y_1$ and $y_2$.
\begin{lemma}
\begin{gather*}
\frac{\partial^2 U_i^\alpha}{\partial\nu_k\partial \nu_l}(\tau_1,0)=\epsilon_i\frac{\partial^2 x_i}{\partial \nu_k\partial\nu_l}(\tau_1,0)\,u_i^\alpha(x_{i,1},0)\\
+\int_0^{\tau_1}\frac{\partial^2 u_i^\alpha}{\partial \nu_k\partial \nu_l}(x_i(t,0),0)
\left|\frac{\partial x_i}{\partial t}\right|dt
\end{gather*}
for $k,l=1,2$, $\alpha=1,2$, and $i\ne j$.
\end{lemma}
\begin{proof}
We differentiate the formula in Lemma~\ref{der-unu} by $\nu_l$ and put $(t,\nu)=(\tau_1,0)$.
Then the first term in the right-hand side becomes the first term of the right-hand side of the above formula,
since $\partial x_i/\partial \nu_k$ vanishes at $(\tau_1,0)$. Also the second term becomes
\begin{equation*}
\int_0^{\tau_1}\frac{\partial}{\partial \nu_l}\left[\frac{\partial u_i^\alpha}{\partial \nu_k}(x_i(t,\nu),\nu)
\left|\frac{\partial x_i}{\partial t}(t,\nu)\right|\right]_{\nu=0}dt\,.
\end{equation*}
This integrand is equal to
\begin{equation*}
\frac{\partial^2 u_i^\alpha}{\partial \nu_k\partial \nu_l}(x_i(t,0),0)
\left|\frac{\partial x_i}{\partial t}\right|+
\epsilon_i(t,\nu)\frac{\partial}{\partial t}\left[\frac{\partial u_i^\alpha}{\partial \nu_k}\frac{\partial x_i}{\partial \nu_l}\right]_{\nu=0}
\end{equation*}
By the same reason as in the proof of Lemma~\ref{der-unu}, we have
\begin{equation*}
\int_0^{\tau_1} \epsilon_i(t,\nu)\frac{\partial}{\partial t}\left[\frac{\partial u_i^\alpha}{\partial \nu_k}\frac{\partial x_i}{\partial \nu_l}\right]_{\nu=0}dt=\frac{\partial u_i^\alpha}{\partial \nu_k}(x_{i,1}0)\left|\frac{\partial x_i}{\partial \nu_l}(\tau_1,0)\right|=0\,.
\end{equation*}
Thus the lemma follows.
\end{proof}
\begin{lemma}\label{lem:ujalphader}
\begin{gather*}
\frac{\partial^2 U_j^\alpha}{\partial \nu_1^2}(\tau_1,0)=\begin{cases}
6\pi A_1'(j_{j,0})\quad&(\alpha=1)\\
6\pi A_1(f_{j,0})\quad &(\alpha=2)
\end{cases}\\
\frac{\partial^2 U_j^\alpha}{\partial \nu_2^2}(\tau_1,0)=\begin{cases}
2\pi A_1'(j_{j,0})\quad&(\alpha=1)\\
2\pi A_1(f_{j,0})\quad &(\alpha=2)
\end{cases}\\
\frac{\partial^2 U_j^\alpha}{\partial \nu_1\partial \nu_2}(\tau_1,0)=0\quad (\alpha=1,2)
\end{gather*}
\end{lemma}
\begin{proof}
A direct computation yields
\begin{gather*}
\frac{\partial^2 U_i^\alpha}{\partial\nu_k\partial\nu_l}(t,\nu)=\frac{\partial^2\theta}{\partial\nu_k\partial\nu_l}(t,\nu)\,u_j^\alpha(\theta,\nu)
+\frac{\partial\theta}{\partial\nu_k}(t,\nu)\left(\frac{\partial u_j^\alpha}{\partial \nu_l}+\frac{\partial u_j^\alpha}{\partial \theta}\frac{\partial\theta}{\partial\nu_l}\right)\\
+\frac{\partial\theta}{\partial\nu_l}(t,\nu)\frac{\partial u_j^\alpha}{\partial \nu_k}(\theta,\nu)
+\int_0^{\theta(t,\nu)
}\frac{\partial^2 u_j^\alpha}{\partial\nu_k\partial \nu_l}(\theta,\nu)\,d\theta\,.
\end{gather*}
Since 
\begin{equation*}
u_j^\alpha(\theta,\nu)=A_1(f_j)(f_j-f_{j,0})^\alpha,\quad f_j=f_{j,0}+\nu_1(1-\cos\theta)+\nu_2\sin\theta,
\end{equation*}
it is easily seen that, for $\alpha=1,2$,
the functions
\begin{equation*}
u_j^\alpha,\quad \frac{\partial u_j^\alpha}{\partial \nu_k},\quad \frac{\partial u_j^\alpha}{\partial\nu_l},\quad \frac{\partial u_j^\alpha}{\partial \theta}
\end{equation*}
vanish at $(\theta,\nu)=(2\pi,0)$.
Since $\theta(\tau_1,0)=2\pi$, we therefore obtain
\begin{equation*}
\frac{\partial^2 U_j^\alpha}{\partial\nu_k\partial\nu_l}(\tau_1,0)=\int_0^{2\pi}
\frac{\partial^2 u_j^\alpha}{\partial\nu_k\partial \nu_l}(\theta,0)\,d\theta\,.
\end{equation*}
From this formula the lemma follows immediately.
\end{proof}
\begin{lemma}
\begin{equation*}
\frac{\partial^2 y_\alpha}{\partial\nu_k\partial \nu_l}(\tau_1,0)=\sum_{i\ne j}\epsilon_i\frac{\partial^2 x_i^\alpha}{\partial\nu_k\partial\nu_l}(\tau_1,0)\,u_i^\alpha(x_{i,1},0)\quad (\alpha=1,2)
\end{equation*}
\end{lemma}
\begin{proof}
First, note that the sum in the right-hand side is equal to the sum in such $i$ that $i\in I$, since for $i\ne j$ with $i\not\in I$ the value $f_i(x_{i,1})$ is equal to some $b_k$ $(k\ne j,j-1)$ and $u_i^\alpha(x_{i,1},0)=0$
in this case. By \eqref{eq:yalpha} we have
\begin{equation*}
\frac{\partial y_\alpha}{\partial \nu_k}(t,\nu)=
\sum_{i\in I}\epsilon_i\frac{\partial x_i^\alpha}{\partial\nu_k}(t,\nu)\,u_i^\alpha(x_{i}(t,\nu),0)
\end{equation*}
Noting the fact that $\partial x_i/\partial \nu_k$ vanishes at $(\tau_1,0)$, we obtain the lemma by differentiating this formula with $\nu_l$. 
\end{proof}

From the above lemmas and the formula
\begin{equation*}
\sum_{i\ne j}\frac{\partial^2 U_i^\alpha}{\partial\nu_k\partial\nu_l}(\tau_1,0)+\frac{\partial^2 U_j^\alpha}{\partial\nu_k\partial\nu_l}(\tau_1,0)
=0\,,
\end{equation*}
we have
\begin{equation}\label{yalphader}
\frac{\partial^2 y_\alpha}{\partial\nu_k\partial \nu_l}(\tau_1,0)=-\sum_{i\ne j}\int_0^{\tau_1}\frac{\partial^2 u_i^\alpha}{\partial \nu_k\partial \nu_l}(x_i(t,0),0)
\left|\frac{\partial x_i}{\partial t}\right|dt
-\frac{\partial^2 U_j^\alpha}{\partial\nu_k\partial\nu_l}(\tau_1,0)\,.
\end{equation}
Thus we need to compute the integrals in the right-hand
side of the above formula. The following lemma is straightforward.
\begin{lemma}\label{lem:uialphader}
\begin{gather*}
\frac{\partial^2 u_i^\alpha}{\partial \nu_1^2}(x_i,0)=
\frac{3u_i^\alpha(x_i,0)}{(f_i-f_{j,0})^2},\quad
\frac{\partial^2 u_i^\alpha}{\partial \nu_2^2}(x_i,0)=
\frac{u_i^\alpha(x_i,0)}{(f_i-f_{j,0})^2},\\
\frac{\partial^2 u_i^\alpha}{\partial\nu_1\partial \nu_2}(x_i,0)=0 \qquad (\alpha=1,2)\,.
\end{gather*}
\end{lemma}
\begin{cor}\label{cor:derynu2}
\begin{gather*}
\frac{\partial^2 y_1}{\partial\nu_1^2}(\tau_1,0)=-3C,\quad \frac{\partial^2 y_1}{\partial\nu_2^2}(\tau_1,0)=-C,\\
\frac{\partial^2 y_1}{\partial\nu_1\partial\nu_2}(\tau_1,0)=0,\quad \frac{\partial^2 y_2}{\partial\nu_k\partial\nu_l}(\tau_1,0)=0\quad (k,l=1,2)\,,
\end{gather*}
where $C$ is the constant given in Lemma~\ref{der-thetanu}.
\end{cor}
\begin{proof}
First we consider the case where $\alpha=2$ and $k=l$. By 
\eqref{yalphader} and Lemmas \ref{lem:ujalphader} and
\ref{lem:uialphader}, we have
\begin{equation*}
\frac{\partial^2 y_2}{\partial \nu_k^2}(\tau_1,0)=-e\left[\sum_{i\ne j}U_i^0(\tau_1,0)+U_j^0(\tau_1,0)\right]=0\,,
\end{equation*}
where $e=3$ or $1$ according as $k=1$ or $2$ respectively. Similarly, for the case where $\alpha=1$,
we have
\begin{align*}
\frac{\partial^2 y_1}{\partial \nu_k^2}(\tau_1,0)=&-e\left[\sum_{i\ne j}\int_0^{\tau_1}\frac{u_i^0(x_i(t,0),0)}{f_i(x_i(t,0))-f_{j,0}}\left|\frac{\partial x_i}{\partial t}\right|\,dt+2\pi A_1'(f_{j,0})\right]\\
=&-eC\,,
\end{align*}
Also, for the case where $k\ne l$, we have
\begin{equation*}
\frac{\partial y_\alpha}{\partial u_k\partial u_l}(\tau_1,0)=0\quad (\alpha=1,2)\,.
\end{equation*}
Thus the corollary follows.
\end{proof}

By the formulas \eqref{der-t12}, \eqref{y0t}, \eqref{der-tnu2}, \eqref{der-tnu1}, \eqref{der-y0tnu},
and by Lemma~\ref{lem:1stdernu}, Corollary~\ref{cor:y0der2},
Corollary~\ref{cor:derynu2}, we obtain the formulas in
Proposition~\ref{mainprop} by putting
\begin{equation*}
c=-\frac{C}2,\qquad c'=A_1(f_{j,0})\frac{\partial\theta}{\partial t}(\tau_1,0)\,.
\end{equation*}
It is clear that $c'\ne 0$. Therefore, to complete the
proof of the proposition, it is enough to show the 
following lemma.
\begin{lemma}
$C> 0$.
\end{lemma}
\begin{proof}
We put
\begin{equation*}
W_i(t,\nu)=U_i^1(t,\nu)+(f_{j,0}-a_n)U_i^0(t,\nu)\,.
\end{equation*}
Then we have
\begin{equation*}
\sum_{i=1}^n W_i(t,\nu)=0\,,
\end{equation*}
\begin{equation*}
W_i(t,\nu)=\int_0^{t}\frac{-(f_i-a_n)\sqrt{(-1)^{i-1}\prod_{k\ne j,j-1}(f_i-b_k)}\,|\frac{\partial x_i(t,\nu)}{\partial t}|\,dt}
{\sqrt{(f_i-f_{j,0})^2-2\nu_1(f_i-f_{j,0})-\nu_2^2}}
\end{equation*}
for $i\ne j$ and
\begin{equation*}
W_j(t,\nu)=\int_0^{\theta(t,\nu)}A_1(f_j)(f_j-a_n)d\theta\,.
\end{equation*}
We now take the derivative in $\nu_1$ and put $\nu=0$, $t=\tau_1$: Since
\begin{equation*}
\frac{\partial W_j}{\partial \nu_1}(\tau_1,0)=2\pi\frac{\partial}{\partial\lambda}\left(A_1(\lambda)(\lambda-a_n)\right)\vert_{\lambda=f_{j,0}}+\frac{\partial\theta}{\partial\nu_1}(\tau_1,0)A_1(f_{j,0})(f_{j,0}-a_n)\,,
\end{equation*}
and since
\begin{equation*}
\frac{\partial\theta}{\partial\nu_1}(\tau_1,0)A_1(f_{j,0})=-C\,,
\end{equation*}
we have
\begin{equation}\label{eq:C}
C(f_{j,0}-a_n)=\sum_{i\ne j}\frac{\partial W_i}{\partial \nu_1}(\tau_1,0)+2\pi\frac{\partial}{\partial\lambda}\left(A_1(\lambda)(\lambda-a_n)\right)\vert_{\lambda=f_{j,0}}\,.
\end{equation}

On the other hand, in view of the formula \eqref{basiceq} we have
\begin{gather}\label{eq1}
\frac{\partial W_i}{\partial \nu_1}(t_i,0)=
\int_0^{t_i}\frac{-(f_i-a_n)\sqrt{(-1)^{i-1}\prod_{k\ne j,j-1}(f_i-b_k)}\,|\frac{\partial x_i(t,\nu)}{\partial t}|\,dt}
{|f_i-f_{j,0}|(f_i-f_{j,0})}\\ 
=\lim_{b_j,b_{j-1}\to f_{j,0}}\frac{\partial }{\partial b_j}
\int_{a_i^+}^{a_{i-1}^-}
\frac{(-1)^{l}A(\lambda)\,(\lambda-a_n)\,G(\lambda)\ d\lambda}
{\sqrt{-\prod_{k=1}^{n-1}(\lambda-b_k)
\cdot\prod_{k=0}^n(\lambda-a_k)}}\quad(i\ne j)\,,
\label{eq2}\end{gather}
where $G(\lambda)=\prod_{k\ne j,j-1}(\lambda-b_k)$
and $t_i$ is the time defined in \S5.
Also, when $\nu\ne 0$ $(b_j\ne b_{j-1})$, we have
\begin{align*}
W_j(t_j,\nu)=&\int_{b_j}^{b_{j-1}}
\frac{A(\lambda)(\lambda-a_n)\sqrt{(-1)^j\prod_{k\ne j,j-1}(\lambda-b_k)}\,d\lambda}{\sqrt{(-1)^j
\prod_{l=0}^n(\lambda-a_l)}\sqrt{(b_{j-1}-\lambda)(\lambda-b_j)}}\\
=&2\int_0^{\pi}A_1(f_j)(f_j-a_n)\,d\theta\,, 
\end{align*}
where $f_j=f_{j,0}+\nu_1(1-\cos\theta)+\nu_2\sin\theta$. Therefore,
\begin{equation}
\begin{gathered}
2\pi\frac{\partial}{\partial\lambda}\left(A_1(\lambda)(\lambda-a_n)\right)\vert_{\lambda=f_{j,0}}=\frac{\partial W_j}{\partial \nu_1}(t_j,0)=\\
\lim_{b_j\to f_{j,0}-0}\frac{\partial}{\partial b_j}\int_{b_j}^{b_{j-1}}
\left.\frac{A(\lambda)(\lambda-a_n)\sqrt{(-1)^j\prod_{k\ne j,j-1}(\lambda-b_k)}\,d\lambda}{\sqrt{(-1)^j
\prod_{l=0}^n(\lambda-a_l)}\sqrt{(b_{j-1}-\lambda)(\lambda-b_j)}}\right\vert_{b_{j-1}=f_{j,0}}.
\end{gathered}
\end{equation}
Then, taking the limit $b_{j-1}=f_{j,0}, b_j\to
f_{j,0}-0$ (or $\nu_2=0, \nu_1\to -0$) in the inequality of Proposition~\ref{prop:newineq} (2),
we have by Proposition~\ref{prop:limitineq}
\begin{equation}\label{ineq:wi}
\sum_{i\ne j}\frac{\partial W_i}{\partial \nu_1}
(t_i,0)+\frac{\partial W_j}{\partial\nu_1}(t_j,0)> 0\,.
\end{equation}

Since $t=\tau_1$ is the first zero
of the Jacobi field $Z_j(t)$ (and $Z_{j-1}(t)$),
it follows that $t_j=\tau_1$, $t_i\le\tau_1$ for
$i>j$, and $t_i\ge\tau_1$ for $i<j$ by Proposition~\ref{prop:main}.
From the formula \eqref{eq1} it can be easily seen
that the integrand of that formula is positive
when $i>j$ and is negative when $i<j$.
Therefore we have
\begin{equation*}
\frac{\partial W_i}{\partial\nu_1}(\tau_1,0)\ge
\frac{\partial W_i}{\partial\nu_1}(t_i,0)\qquad
(i\ne j)\,.
\end{equation*}
Thus the assertion follows from \eqref{eq:C} and \eqref{ineq:wi}.
%If $t=\tau_1$ is the $k$-th zero of the Jacobi field $Z_j(t)$ $(k\ge 2)$, then we put the zeros of the Jacobi field $Z_j(t)$ as $0=r_j^0<r_j^1<\dots<r_j^k=\tau_1$. Regarding each $r_j^l$ ($0\le l\le k-1$) as the starting point, we see from the above argument that 
%\begin{gather*}
%\frac{\partial W_i}{\partial\nu_1}(r_j^{l+1},0)-\frac{\partial W_i}{\partial\nu_1}(r_j^{l},0)>\frac{\partial W_i}{\partial\nu_1}(t_i,0)\quad(i\ne j)\,,\\
%\frac{\partial W_j}{\partial\nu_1}(r_j^{l+1},0)-\frac{\partial W_j}{\partial\nu_1}(r_j^{l},0)=\frac{\partial W_j}{\partial\nu_1}(t_j,0)\quad (i\ne j)
%\end{gather*}
%Thus we have
%\begin{gather*}
%C(f_{j,0}-a_n)=\sum_{i\ne j}\frac{\partial W_i}{\partial \nu_1}(\tau_1,0)+\frac{\partial W_j}{\partial \nu_1}(\tau_1,0)\\
%>k\left(\sum_{i\ne j}\frac{\partial W_i}{\partial \nu_1}(t_j,0)+\frac{\partial W_j}{\partial\nu_1}(t_j,0)\right)\ge 0\,.
%\end{gather*}
\end{proof}

This finishes the proof of Proposition~\ref{mainprop}.

%%%%

\end{document}